\documentclass[10pt]{article}

\usepackage[french,english]{babel}
\usepackage[latin1]{inputenc}
\usepackage[T1]{fontenc}
\usepackage{array}
\usepackage{amsfonts}
\usepackage{amstext}
\usepackage{amsmath}
\usepackage{amssymb}
\usepackage{amsthm}
\usepackage{graphicx}
\usepackage{textcomp}
\usepackage{graphics,xcolor}
\usepackage{stmaryrd}

\newcommand\ds{\displaystyle}
\newcommand\tn{\textnormal}

\def\fcar{{\mathbf 1}}

\newtheorem{theorem}{Theorem}[section]

\newtheorem{corollary}[theorem]{Corollary}
\newtheorem{proposition}[theorem]{Proposition}
\newtheorem{lemma}[theorem]{Lemma}
\newtheorem{remark}[theorem]{Remark}
\newtheorem{algorithm}{Algorithm}

\numberwithin{equation}{section}

\title{Global Carleman estimates for waves and applications\footnote{Partially supported by the Agence Nationale de la Recherche (ANR, France), Project CISIFS number NT09-437023 and the University Paul Sabatier (Toulouse 3), AO PICAN.}}
\author{
Lucie Baudouin$^{1,2,}$\footnote{e-mail: {\tt baudouin@laas.fr}}\\
{\it\footnotesize $^{1}$ CNRS, LAAS, 7 avenue du colonel Roche, F-31400 Toulouse, France;}\\
{\it\footnotesize $^{2}$ Univ de Toulouse, LAAS, F-31400, Toulouse, France;}\\
Maya de Buhan$^{3,}$\footnote{e-mail: {\tt maya.de-buhan@parisdescartes.fr}}\\
{\it\footnotesize $^{3}$ CNRS, UMR 8145, MAP5, Université Paris Descartes, Sorbonne Paris Cité, France;
}\\
Sylvain Ervedoza$^{4,5,}$\footnote{e-mail: {\tt ervedoza@math.univ-toulouse.fr}}\\
{\it\footnotesize $^{4}$ CNRS, Institut de Math\'ematiques de Toulouse UMR 5219 ; F-31400 Toulouse, France;}\\
{\it\footnotesize $^{5}$ Univ de Toulouse, IMT, F-31400 Toulouse, France.}
}

\begin{document}

\maketitle

\abstract{In this article, we extensively develop Carleman estimates for the wave equation and give some applications. We focus on the case of an observation of the flux on a part of the boundary satisfying the Gamma conditions of Lions. We will then consider two applications. The first one deals with the exact controllability problem for the wave equation with potential. Following the duality method proposed by Fursikov and Imanuvilov in the context of parabolic equations, we propose a constructive method to derive controls that weakly depend on the potentials. The second application concerns an inverse problem for the waves that consists in recovering an unknown time-independent potential from a single measurement of the flux. In that context, our approach does not yield any new stability result, but proposes a constructive  algorithm to rebuild the potential. 
In both cases, the main idea is to introduce weighted functionals that contain the Carleman weights and then to take advantage of the freedom on the Carleman parameters to limit the influences of the potentials.
}
\\


\section{Introduction}

The goal of this article is to revisit observability properties in the light of Carleman estimates for the wave equation with a potential in a bounded domain. We will present applications of the appropriate Carleman estimates in two directions:
\begin{itemize}
	\item In control theory on the dependence of the exact controls for waves with respect to the potentials;
	\item In an inverse problem for the wave equation in which the potential is unknown, where we will give a reconstruction algorithm for the potential.
\end{itemize}

\subsection{Setting}
Let $\Omega$ be a smooth bounded domain of $\mathbb{R}^n$, $n\geq 1$, and $T>0$. We consider the wave equation 
\begin{equation}\label{eqz00}
	\left\{ \begin{array}{ll}
 		\partial_t^2 z -\Delta z+p z =g,  \qquad&  \tn{in } \Omega \times (0,2T),\\
		 z =0,  &\text{on } \partial \Omega \times (0,2T),\\
		z(0)= z_0,  \quad \partial_t z(0)= z_1,  & \tn{in } \Omega.
	\end{array}\right.
\end{equation}
Here, $z$ denotes the amplitude of the waves, $p$ is a potential supposed to be in $L^\infty(\Omega\times (0,2T))$, $g$ is a source term for instance in $L^2(\Omega \times (0,2T) )$ and $(z_0,z_1)$ are the initial data lying in $H^1_0(\Omega) \times L^2(\Omega)$. It is by now well-known that, due to hidden regularity results \cite{Lions}, under these assumptions, the normal derivative of $z$ on the boundary belongs to $L^2(\partial \Omega \times (0,2T))$. 

In this article, we focus on the following observability property: 
\begin{quote}
Given $\Gamma_0 \subset \partial \Omega$, can we determine $z$ solution of \eqref{eqz00} from the knowledge of $g$, $p$ and of $\partial_\nu z$ on $\Gamma_0 \times (0,2T)$? 
\end{quote}

When $p= 0$, this question has a positive answer if and only if the Geometric Control Condition holds \cite{Bardos,BurqGerard} for $\Omega, 2T$ and $\Gamma_0$. Roughly speaking, it asserts that all the rays of geometric optics in $\Omega$, which here are simply straight lines reflected on the boundary according to Descartes Snell's law, should meet the observation region $\Gamma_0$ at a non-diffractive point in a time less than $2T$.

However, other methods exist based on multiplier techniques  \cite{Lions,KomornikMultipliers} or on Carleman estimates \cite{FursikovImanuvilov,Zhang00}. These methods use stronger geometrical assumptions, and in particular the following ones, sometimes referred to as the Gamma-condition of Lions or the multiplier condition.\\

\noindent{\bf Geometric and time conditions:} 
\begin{equation}
	\label{GCC-multiplier}
	\exists \, x_0 \not \in \overline{\Omega} , \tn{ such that } \Gamma_0 \supset \{x \in \partial \Omega, \ (x-x_0) \cdot \nu(x) \geq 0 \}, 
\end{equation}
\begin{equation}
	\label{GCC-Time}
	T > \sup_{x \in \Omega} | x - x_0|.
\end{equation}

The advantage of Carleman estimates on the multiplier techniques is that they allow to easily handle potentials in $L^\infty(\Omega\times (0,2T))$ - see e.g. \cite{FursikovImanuvilov,Zhang00,DuyckaertsZhangZuazua}. In the applications we have in mind and that will be developed hereafter, it will be important to understand the dependence of the observability inequalities with respect to the potentials. This precisely explains why the path we have chosen hereafter uses Carleman estimates.

In order to state our results precisely, we shall need several notations. To make them easier, instead of working on $(0,2T)$ as in \eqref{eqz00}, we do a translation in time in order to consider:
\begin{equation}\label{eqz0}
	\left\{ \begin{array}{ll}
 		\partial_t^2 z -\Delta z+p z =g,  \qquad&  \tn{in } \Omega \times (-T,T),\\
		 z =0,  &\text{on } \partial \Omega \times (-T,T),\\
		z(-T)= z_0^{-T},  \quad		 \partial_t z(-T)= z_1^{-T},  & \tn{in } \Omega.
	\end{array}\right.
\end{equation}
Let us define, once for the whole paper, the weight functions we shall consider in Carleman estimates.\\

\noindent{\bf Weight functions:} Assume that $\Gamma_0$ satisfies \eqref{GCC-multiplier} for some $x_0 \not\in \overline\Omega$. Let $\beta \in (0,1)$, and define, for $(x,t) \in \Omega \times (-T,T)$,
	\begin{equation}
		\label{poids}
		\psi(x,t) = |x-x_0|^2-\beta t^2+C_0, \quad \text{ and  for $\lambda>0$,} \quad
		\varphi(x,t) = e^{\lambda \psi(x,t)},
	\end{equation}
	where $C_0>0$ is chosen such that $\psi\geq 1$ in $\Omega\times(-T,T)$. 
\medskip

Note that the weight function $\varphi$ defined that way depends on $\beta \in (0,1)$ and $\lambda >0$ and shall rather be denoted $\varphi_{\beta, \lambda}$, but these dependences are omitted for simplifying notations. 

We also define, for $m >0$, the spaces
$$
	\begin{array}{c}
		\ds
	L^\infty_{\leq m} (\Omega) = \{q \in L^\infty(\Omega), \|q\|_{L^\infty(\Omega)} \leq m \},
	\smallskip\\
		\ds 
	L^\infty_{\leq m} (\Omega\times (-T,T)) = \{p \in L^\infty(\Omega\times (-T,T)), \|p\|_{L^\infty(\Omega\times (-T, T))} \leq m \}.
	\end{array}
$$

The main results we shall use are the following ones:

\begin{theorem}
	\label{ThmCarleman}
	Assume the multiplier condition \eqref{GCC-multiplier} and the time condition  \eqref{GCC-Time}. Let $\beta \in (0,1)$ be such that
	\begin{equation}
		\label{GCC-Time-Beta}
		\sup_{x\in \Omega}|x-x_0| < \beta T.
	\end{equation}
	Then for any $m>0$, there exist $\lambda>0$ independent of $m$, $s_{0}= s_0(m)>0$ and a positive constant $M=M(m)$  
such that for $\varphi$ being defined as in \eqref{poids}, for all $p\in L^\infty_{\leq m} (\Omega\times (-T,T))$ and for all $s\geq s_{0}$:
\begin{equation}\label{CarlemTp}	
	\begin{aligned}	
		s \int_{-T}^{T} \int_{\Omega} &e^{2s\varphi}\left(|\partial_t z|^2 + |\nabla z |^2 \right)\,dxdt 
		+ s^3\int_{-T}^{T} \int_{\Omega} e^{2s\varphi}|z|^2\,dxdt\\
		+s  \int_{\Omega} &e^{2s\varphi(-T)}\left(|\partial_t z(-T)|^2 + |\nabla z (-T)|^2\right)dx
		+ s^3 \int_{\Omega} e^{2s\varphi(-T)}|z(-T)|^2\,dx \\
		&\leq M\int_{-T}^{T} \int_{\Omega} e^{2s\varphi}|\partial_t^2 z -\Delta z+p z |^2\,dxdt 
		+ Ms \int_{-T}^{T} \int_{\Gamma_0} e^{2s\varphi} \left|\partial_\nu z\right|^2\,d\sigma dt ,
		\end{aligned}
\end{equation}
for all $z\in L^2(-T,T;H_0^1(\Omega))$ satisfying $\partial_t^2 z -\Delta z+p z \in L^2(\Omega\times (-T,T))$ and $\partial_\nu z \in L^2(\Gamma_0 \times (-T,T))$.
\end{theorem}

\begin{theorem}
	\label{ThmCarleman-t=0}
	Under the assumptions of Theorem \ref{ThmCarleman}, if $z$ furthermore satisfies $z(\cdot,0) = 0$ in $\Omega$, one also has
	\begin{multline}
		s^{1/2} \int_{\Omega} e^{2s\varphi(0)}|\partial_t z(0)|^2 \,dx 
		\leq M\int_{-T}^{T} \int_{\Omega} e^{2s\varphi}|\partial_t^2 z -\Delta z+p z |^2\,dxdt 
		\\
		+ Ms \int_{-T}^{T} \int_{\Gamma_0} e^{2s\varphi} \left|\partial_\nu z\right|^2\,d\sigma dt.
		\label{Carlem0p}
	\end{multline}
	In particular, if $z(\cdot,0) = 0$ in $\Omega$ and $q \in L^\infty_{\leq m}(\Omega)$, then for all $z\in L^2(0,T;H_0^1(\Omega))$ satisfying $\partial_t^2 z -\Delta z+q z \in L^2(\Omega\times (0,T))$ and $\partial_\nu z \in L^2(\Gamma_0 \times (0,T))$ and for all $s \geq s_0(m)$,
	\begin{multline}\label{CarlemT-0}
		s^{1/2}  \int_{\Omega} 
		 e^{2s\varphi(0)}|\partial_t z(0)|^2 \,dx		
		+ s \int_{0}^{T} \int_{\Omega} 
		e^{2s\varphi}\left(|\partial_t z|^2 + |\nabla z |^2 \right)\,dxdt 
		+ s^3\int_{0}^{T} \int_{\Omega} e^{2s\varphi}|z|^2\,dxdt \\
		\leq M\int_{0}^{T} \int_{\Omega} e^{2s\varphi}|\partial_t^2 z -\Delta z+q z|^2\,dxdt 
		 + Ms \int_{0}^{T} \int_{\Gamma_0} e^{2s\varphi} \left|\partial_\nu z\right|^2\,d\sigma dt .
	\end{multline}
\end{theorem}

Note that the condition $z(\cdot, 0) = 0$ in $\Omega$ of Theorem \ref{ThmCarleman-t=0} makes sense for $z$ such that $z\in L^2((0,T);H_0^1(\Omega))$ and $\partial_t^2 z -\Delta z+p z \in L^2(\Omega\times (0,T))$ since then $\partial_t^2 z = \Delta z - pz+(\partial_t^2 z -\Delta z+p z)$ belongs to $L^2((0,T);H^{-1}(\Omega))$.

Theorems \ref{ThmCarleman}--\ref{ThmCarleman-t=0} do not claim particular originality and many of their ingredients are already available in the literature, see e.g. \cite{Baudouin01} where very similar estimates are proved and \cite{FursikovImanuvilov,Zhang00,Im02} for more references. However, to our knowledge, this is the first time that these global Carleman estimates are written under that form, which is easier to use to achieve our goals. Detailed proofs of Theorem \ref{ThmCarleman}--\ref{ThmCarleman-t=0} are given in Section \ref{SecCarleman}.

\subsection{Applications to controllability}\label{SubsecIntroControl}
The idea is to take advantage of the Carleman estimate of Theorem \ref{ThmCarleman} to obtain controls whose dependence with respect to the potential is weak. To be more precise, we focus on the following exact controllability problem: 

\begin{quote}
Given $(y_0^{-T}, y_1^{-T}) \in L^2(\Omega) \times H^{-1} (\Omega)$, find $u \in L^2(\Gamma_0 \times (-T,T))$ such that the solution $y$ of
\begin{equation}
	\label{EqYControl}
	\left\{ \begin{array}{ll}
 		\partial_t^2 y-\Delta y+p y= 0, \quad  & \tn{in } \Omega \times (-T,T),	\\
		y=u \, \fcar_{\Gamma_0},  \qquad &\tn{on } \partial \Omega \times (-T,T),\\
		y(-T)= y_0^{-T},  \quad 
		\partial_t y(-T)= y_1^{-T},  & \tn{in } \Omega
	\end{array}\right.
\end{equation}
solves
\begin{equation}
	\label{ExactControllability}
	y(T) = \partial_t y(T) = 0, \qquad \tn{in } \Omega.
\end{equation}
\end{quote}

This exact controllability problem is equivalent to the observability of the system \eqref{eqz0}. These two properties are dual one from another, as stated by Lions \cite{Lions}  using the Hilbert Uniqueness Method (HUM): the HUM computes the control of minimal $L^2(\Gamma_0 \times (-T, T))$-norm from the minimization of a quadratic functional whose coercivity is equivalent to an observability property for the adjoint system \eqref{eqz0} that can be deduced from Theorem~\ref{ThmCarleman} under the conditions \eqref{GCC-multiplier}--\eqref{GCC-Time}. \\

Actually, we shall not focus on these HUM controls. Nevertheless, the controls that we shall consider below are also computed with a duality argument, based on the ``observability'' inequality \eqref{CarlemTp} directly. Our approach is strongly inspired by the duality strategy employed by Fursikov and Imanuvilov \cite{FursikovImanuvilov}, that has mainly been used for parabolic equations so far. The idea is to minimize, for $s >0$, the functional
\begin{multline}
	\label{K-functional}
	K_{s,p}(z) = \frac{1}{2s}\int^T_{-T}\int_\Omega e^{2s\varphi} | \partial_t^2 z-\Delta z+p z|^2\,dxdt +\frac{1}{2}\int^T_{-T}\int_{\Gamma_0}e^{2s\varphi}  |\partial_\nu z|^2 \,d\sigma dt 
			\\
			+ \langle (y_0^{-T},y_1^{-T}),(z(-T),\partial_t z(-T))\rangle_{(L^2\times H^{-1})\times(H^1_0\times L^2)}, 
\end{multline}
on the trajectories $z$ such that $z\in L^2(-T,T;H^1_0(\Omega))$, $ 	\partial_t^2 z -\Delta z+p z\in L^2(\Omega \times (-T, T))$ and $\partial_{\nu} z \in L^2(\Gamma_0 \times (-T, T))$. Here,  
$$
	\langle (y_0^{-T},y_1^{-T}),(z_0^{-T},z_1^{-T})\rangle_{(L^2\times H^{-1})\times(H^1_0\times L^2)} = \int_\Omega y_0^{-T} z_1^{-T} - \langle y_1^{-T} , z_0^{-T}\rangle_{H^{-1} \times H^1_0} dx,
$$
with
$$
	\langle y_1^{-T} , z_0^{-T}\rangle_{H^{-1} \times H^1_0} = \int_\Omega \nabla (-\Delta_{d})^{-1} y_1^{-T} \cdot  \nabla z_0^{-T}dx, 
$$
where $\Delta_d$ is the Laplace operator with Dirichlet boundary conditions. Note that this functional $K_{s,p}$ depends on $s$ - the parameter chosen in the exponential - and on the potential $p$. 

Then, according to the Carleman inequality of Theorem \ref{ThmCarleman}, under conditions \eqref{GCC-multiplier}--\eqref{GCC-Time}, if, for some $m>0$, $p \in L^\infty_{\leq m} (\Omega \times (-T,T))$ and $s \geq s_0(m)$, we shall easily show that $K_{s,p}$ is strictly convex and coercive. $K_{s,p}$ therefore has a unique minimizer, denoted by $Z[s,p]$ to underline the dependence with respect to the potential $p$ and the parameter $s$. Simple computations prove that if we set 
\begin{equation}
	\label{LinkControlPb-Adjoint}
	Y[s,p] = \frac{1}{s}e^{2 s \varphi} (\partial_t^2 - \Delta +p) Z[s,p] \quad \text{ and } \quad U[s,p] = e^{2 s \varphi} \partial_\nu Z[s,p] \fcar_{\Gamma_0},
\end{equation}
we obtain a solution to the exact controllability problem \eqref{EqYControl}--\eqref{ExactControllability} - see Theorem \ref{PropControle} for precise statements and proofs. We are then in position to study the dependence of the controls with respect to the potential:

\begin{theorem}
	\label{ThmErrorRelative}
	Assume the conditions \eqref{GCC-multiplier}--\eqref{GCC-Time}.
	
	Let $m>0$ and  $p^a$, $p^b$ be two potentials in $L^\infty_{\leq m}(\Omega\times (-T, T))$.
	Given an initial data $(y_0^{-T}, y_1^{-T}) \in L^2(\Omega) \times H^{-1}(\Omega)$, there exists a constant $M = M(m) >0$ independent of $s $ such that the corresponding controlled trajectories $(Y[s, p^a], U[s,p^a])$ and $(Y[s,p^b], U[s,p^b])$ satisfy, for all $s \geq s_0(m)$,
	\begin{multline}
	\label{ErrorRelative}
	\dfrac{
	s\ds \int_{-T}^{T} \int_\Omega e^{-2s\varphi}|Y[s,p^a] - Y[s,p^b]|^2 \,dxdt
	+  
	\int_{-T}^{T} \int_{\Gamma_0} e^{-2s\varphi} |U[s,p^a] - U[s,p^b] |^2\,dxdt 
	}
	{
	s\ds\int_{-T}^{T} \int_\Omega e^{-2s\varphi} (Y[s,p^a]^2 + Y[s,p^b ]^2) \,dxdt	
	+ 
	\int_{-T}^{T} \int_{\Gamma_0} e^{-2s\varphi} (U[s,p^a]^2 + U[s,p^b]^2 )\,dxdt 
	}
	\\
	\leq 
	\frac{M}{s^{3/2}} \|p^a-p^b\|_{L^\infty(\Omega\times (-T,T))} \leq 2m M s^{-3/2},
	\end{multline}
	where $\varphi = \varphi_{\beta, \lambda}$ is chosen so that Theorem \ref{ThmCarleman} holds and $s_0(m)$ is the parameter given by Theorem \ref{ThmCarleman}.
	
	In other words, the relative error between the controlled trajectories $(Y[s,p], U[s,p])$ decays as $s^{-3/2}$ for potentials lying in $L^\infty_{\leq m}(\Omega \times (-T, T))$.
\end{theorem}

Theorem \ref{ThmErrorRelative} states that, as $s$ increases, the control obtained by the minimization of $K_{s,p}$ depends less and less of the potential $p$. This is of course in complete agreement with the results obtained using microlocal analysis, in which the potentials play no role - see e.g. \cite{BurqGerard}. However, to our knowledge, the relations between controls computed for different potentials have not been studied so far. 

The proof of Theorem \ref{ThmErrorRelative} is purely variational and comes from the variational characterization of the controls. The main issue is to track the powers of $s$ during the proof. Again, this strongly relies on the Carleman estimate of Theorem \ref{ThmCarleman}. One will find all the details in Section~\ref{SecCont}.

\subsection{Applications to inverse problems}\label{SubsecIntroInversePb}
The idea now is to take advantage of the Carleman estimate of Theorem \ref{ThmCarleman-t=0} to conceive 
a reconstruction algorithm of the potential from the knowledge of the flux of the solution.
To be more precise, we focus on the following inverse problem:

\begin{quote}
Given the source terms $h$ and $h_{\partial}$ and the initial data $(w_0, w_1)$, 
considering the solution of
\begin{equation}\label{EqW}
	\left\{ \begin{array}{ll}
 		\partial_t^2 W-\Delta W+Q W=h,\qquad   & \tn{in }\Omega \times (0,T),\\
		W =h_{\partial},  & \tn{on } \partial \Omega \times (0,T),\\
		W(0)= w_0, \quad \partial_t W(0)= w_1,  &\tn{in } \Omega,
	\end{array}\right.
\end{equation}
can we determine the unknown potential $Q = Q(x)$, assumed to depend only on $x\in \Omega$,
from the additional knowledge of the flux 
\begin{equation}
	\label{Flux}
	\mu = \partial_{\nu} W,  \quad \hbox{ on } \Gamma_0 \times (0,T)
\end{equation}
of the solution?
\end{quote}

\noindent Under the regularity assumption 
\begin{equation}
	\label{RegAssumptions}
	 W \in H^1((0,T); L^\infty(\Omega)),
\end{equation} 
the positivity condition
\begin{equation}
	\label{Positivity}
	\exists \alpha >0 \tn{ such that }|w_0| > \alpha \tn{ in } \Omega
\end{equation}
and the multiplier conditions \eqref{GCC-multiplier}--\eqref{GCC-Time}, the results in \cite{Baudouin01} (and in \cite{Yam99} under more regularity hypothesis) state the stability of this inverse problem consisting in finding the potential $Q $ from the measurement of the flux \eqref{Flux}. To be more precise, it is proved that, given $Q^a,\ Q^b \in L^\infty_{\leq m}(\Omega)$ and denoting by $W[Q^a]$ and $W[Q^b]$ the corresponding solutions to \eqref{EqW}, there exists a positive constant $M = M(\Omega, T, x_0,m)$ such that
$$
	\label{StabilityResults}
	\frac{1}{M} \|Q^a - Q^b \|_{L^2(\Omega)} 
		\leq 
	\| \partial_t \partial_{\nu} W[Q^a] - \partial_t \partial_\nu W[Q^b] \|_{L^2(0,T; L^2(\Gamma_0))}
		\leq 
	{M} \|Q^a - Q^b \|_{L^2(\Omega)}.
$$

Nevertheless, this stability result 
is not given with a constructive argument that could allow to explain how to find $Q$ from the knowledge of $\partial_{\nu} W  \fcar_{\Gamma_0}$. We are thus interested in deriving an algorithm so that it can eventually be implemented numerically. 

The algorithm we shall propose is based on a data assimilation problem that we briefly present below.\\

Let $m >0$ and $q \in L^\infty_{\leq m} (\Omega)$. Let $\mu \in L^2(\Gamma_0 \times (0,T))$ and $g \in L^2(\Omega \times (0,T))$. We introduce the functional
\begin{equation}
	\label{DefFunctional-J-DataAss}
	J_{s,q}[\mu,g](z) = \frac{1}{2s} \int_0^T \int_{\Omega} e^{2s \varphi} |\partial_{t}^2 z - \Delta z + q z - g|^2\,dxdt
	+ \frac{1}{2} \int_0^T \int_{\Gamma_0} e^{2s \varphi} | \partial_\nu z - \mu |^2 \ d\sigma dt,
\end{equation} 
on the trajectories $z$ such that $z \in L^2(0,T;H_0^1(\Omega))$, $\partial_{t}^2 z - \Delta z + q z \in L^2(\Omega \times (0,T))$, $\partial_\nu z \in L^2(\Gamma_0 \times (0,T))$ and $z(\cdot, 0) = 0$ in $\Omega$. Remark that $z(\cdot, 0)$ makes sense in $H^{-1}(\Omega)$ since $\partial_t^2 z  = \Delta z - q z + (\partial_{t}^2 z - \Delta z + q z)\in L^2(0,T; H^{-1}(0,1))$ under the previous assumptions.\\

Similarly as for $K_{s,p}$, one will see that this functional $J_{s,q}[\mu,g]$ has a unique minimizer $Z$ for $s \geq s_0$ (Proposition \ref{ThmJ-DataAssimilation}). We emphasize that $Z$ depends on $s$ and $q$, but the context will make it obvious and we therefore drop these dependences to simplify the notations. More importantly, we can study how the minimizer $Z$ depends on $g \in L^2 (\Omega\times (0,T))$ and, similarly as in Theorem \ref{ThmErrorRelative}, we will prove the following result:
\begin{theorem}
	\label{ThmDependenceG}
	Assume the multiplier condition \eqref{GCC-multiplier} and the time condition \eqref{GCC-Time}.
	
	Assume that $\mu \in L^2(\Gamma_0 \times (0,T))$ and $g^a,\,  g^b \in L^2(\Omega \times (0,T))$. Let $m>0$ and $q \in L^\infty_{\leq m} (\Omega)$. 
	Let $Z^j$ be the unique minimizer (see Proposition \ref{ThmJ-DataAssimilation}) of the functionals $J_{s,q}[\mu, g^j]$ for $j \in \{a, b\}$. Then there exist positive constants $s_0(m)$ and $M = M(m)$ such that for $s \geq s_0(m)$ we have:
	\begin{equation}
		\label{EstMin-s}
		s^{1/2}\int_\Omega e^{2 s \varphi(0)} |\partial_t Z^a(0) -\partial_t Z^b(0)|^2  \,dx \leq M \int_0^T \int_\Omega e^{2 s\varphi} |g^a - g^b|^2 \,dxdt, 
	\end{equation}
	where $\varphi = \varphi_{\beta, \lambda}$ is chosen so that Theorem \ref{ThmCarleman-t=0} holds and $s_0(m)$ is the parameter in Theorem \ref{ThmCarleman-t=0}.
\end{theorem}

Based on this result, we propose an algorithm to compute the potential $Q$ from the measurement of the flux, based on some additional knowledge on the $L^\infty(\Omega)$-norm of the unknown potential $Q$:
\begin{equation}
	\label{AssumptionQ}
	\exists m>0 \quad \tn{such that} \quad Q \in L^\infty_{\leq m} (\Omega).
\end{equation}
The algorithm then is the following:
\begin{algorithm}\label{Algo}
\begin{itemize}
	\item Initialization: $q^0 = 0$.
	\item Iteration. Given $q^k$, we set $\mu^k = \partial_t \left(\partial_\nu w[q^k] - \partial_\nu W[Q]\right)$ on $\Gamma_0 \times (0,T)$, where $w[q^k]$ denotes the solution of  
\begin{equation}\label{Eqwk}
	\left\{ \begin{array}{ll}
 		\partial_t^2 w-\Delta w+q^k w=h,\qquad   & \tn{in }\Omega \times (0,T),\\
		w =h_{\partial},  & \tn{on } \partial \Omega \times (0,T),\\
		w(0)= w_0, \quad \partial_t w(0)= w_1,  &\tn{in } \Omega,
	\end{array}\right.
\end{equation}
	corresponding to \eqref{EqW}  with the potential $q^k$. We then introduce the functional $J_{s,q^k}[\mu^k,0]$ defined, for some $s>0$ that will be chosen independently of $k$, by
		\begin{multline}
			\label{FunctionalJ-k}
			J_{s,q^k}[\mu^k,0](z ) 
			=  \frac{1}{2s} \int_0^T \int_{\Omega} e^{2s \varphi} |\partial_{t}^2 z - \Delta z + q^k z |^2\,dxdt
			\\
			+ \frac{1}{2} \int_0^T \int_{\Gamma_0} e^{2s \varphi} | \partial_\nu z - \mu^k |^2 \ d\sigma dt,
		\end{multline}	
		on the trajectories $z \in L^2(0,T; H^1_0(\Omega))$ such that  $\partial_{t}^2 z - \Delta z + q^k z \in L^2(\Omega \times (0,T))$, $\partial_\nu z \in L^2(\Gamma_0 \times (0,T))$ and $z(\cdot, 0) = 0$ in $\Omega$.\\
		Let $Z^k$ be the unique minimizer (see Proposition \ref{ThmJ-DataAssimilation}) of the functional $J_{s,q^k}[\mu^k,0]$, and then set
		\begin{equation}
			\label{Q-k-tilde}
			\tilde q^{k+1} = q^k + \frac{\partial_t Z^k(\cdot, 0)}{w_0},
		\end{equation}
		where $w_0$ is the initial condition in \eqref{EqW}.
		
		Finally, set
		\begin{equation}
			\label{Q-k+1}
			q^{k+1} = T_m (\tilde q^{k+1}), \quad \tn{ where } T_m(q)= \left\{ \begin{array}{ll} q, &\tn{ if } |q| \leq m, \\ \tn{sign}(q) m, &\tn{ if } |q| \geq m. \end{array}\right. 
		\end{equation}
\end{itemize}
\end{algorithm}

One will see in Section~\ref{SecInversePb} how Theorem \ref{ThmDependenceG} allows to prove the convergence of the above algorithm for $s$ large enough: 

\begin{theorem}\label{Thm-Convergence}
Assuming the multiplier and time conditions \eqref{GCC-multiplier}--\eqref{GCC-Time} and \eqref{RegAssumptions},\eqref{Positivity} and \eqref{AssumptionQ}, there exists a constant $M>0$ such that for all $s\geq s_0(m)$ and $k \in \mathbb{N}$,
\begin{equation}
	\label{EstimeesConvergence}
	\int_\Omega e^{2 s \varphi(0)} ( q^{k+1} - Q)^2 \, dx \leq \frac{M}{\sqrt{s}}  \int_{\Omega} e^{2 s \varphi(0)} (q^k - Q)^2 \, dx.
\end{equation}
In particular, when $s$ is large enough, the above algorithm converges.
\end{theorem}

We emphasize here that this approach is constructive. In particular, at each step of this algorithm, we solve a quadratic strictly convex minimization problem, which can be easily done.  Besides, the algorithm necessarily converges to $Q$ due to Theorem \ref{Thm-Convergence}. This is a great advantage compared to the classical methods for solving this inverse problem, that usually consists in minimizing
\begin{equation}
	\label{DirectApproach}
	J(q)= \int_0^T \int_{\Gamma_0} | \partial_\nu w[q] - \partial_\nu W[Q]|^2 d \sigma dt,
\end{equation}
$w[q]$ being the solution of \eqref{EqW} (or \eqref{Eqwk}) corresponding to the potential $q$, which is not convex and may have several local minima. Of course, due to that fact, it is very difficult to propose a convergence result based on the minimization of the functional $J$ in \eqref{DirectApproach}, since classical minimization algorithms are not guaranteed to converge toward the global minimum of $J$.\\

However, when doing numerics, as underlined from the seminal work \cite{GL} to the most recent developments \cite{ErvZuaCime}, the discretization process usually creates  spurious high-frequency waves that do not travel and strongly disturb control processes, even making the discrete controls blow up for some initial data to be controlled. Based on the recent work \cite{BaudouinErvedoza11} that proves Carleman estimates for discrete waves uniformly with respect to the space discretization parameter, we will investigate the numerical methods to compute approximations of potentials in a future work. \\

Let us finally conclude this section by giving some references considering inverse problems for hyperbolic equations using Carleman estimates.

The use of Carleman estimates to prove uniqueness results in inverse problems was introduced in \cite{BuKli81} by Bukhge{\u\i}m and Klibanov.  
The first proofs of the stability of inverse problems for hyperbolic equations rely on uniqueness results obtained by local Carleman estimates (see e.g. \cite{Im02,Isakov}). One can read for instance \cite{PuelYam96,PuelYam97,Yam99,YamZhang03,StefanovUhlmann2011}, where the method uses compactness-uniqueness arguments based on observability inequalities. 

Concerning other inverse problems for the wave equation with a single observation, the references \cite{ImYamIP01,ImYamCom01,ImYamIP03} consider the case of interior or Dirichlet boundary data observation and use global Carleman estimates, as in \cite{Baudouin01}. 

For more generic hyperbolic models, one could also mention \cite{BaudouinCrepeauValein11,BMO07,ImYamCOCV05} giving stability of inverse problem from appropriate global Carleman estimates respectively for a network of 1-d strings, a discontinuous wave equation or the Lam\'e system. Let us also mention the work \cite{BellassouedIP04} for logarithmic stability results when no geometric condition is fulfilled.

\subsection{Outline and notations}
The paper is organized as follows. Section~\ref{SecCarleman} is devoted to the proof of several weighted estimates yielding the Carleman estimates of Theorems~\ref{ThmCarleman} and \ref{ThmCarleman-t=0}. In Section~\ref{SecCont}, we show how the result of Theorem~\ref{ThmCarleman} can be used to solve the exact controllability problem related to equation~(\ref{EqYControl}). In particular, we give the proof of Theorem~\ref{ThmErrorRelative} on the dependence of the control with respect to the potential. Section~\ref{SecInversePb} then focuses on the application of Theorem~\ref{ThmCarleman-t=0} to the inverse problem related to equation~(\ref{EqW}) of recovering the unknown potential. We describe in detail the algorithm proposed for solving this inverse problem and prove its convergence as stated by Theorem~\ref{Thm-Convergence}.\\

\noindent Due to the important number of notations, let us make precise some of them:
\begin{itemize}
	\item $y$ represents controlled trajectories and $u$ stands for the controls;
	\item $z$ represents free trajectories of the waves with a source term and satisfying homogeneous Dirichlet boundary conditions;
	\item $v$ represents trajectories of the waves satisfying $v(\pm T ) = \partial_t v(\pm T) = 0$;
	\item $w$ represents trajectories of the waves conjugated by the Carleman weight in Section \ref{SecCarleman};
	 	In the rest of the article, it represents solutions of \eqref{Eqwk};
	\item $\varphi, \, \psi$ are the Carleman weights defined by \eqref{poids};
	\item $s, \lambda, \beta$ are the parameters entering into the Carleman estimates; 
	\item $p$, $q$, $Q$ denote potentials, but potentials denoted by $p$ may depend on $x $ and $t$ whereas potentials $q$ depend only on $x$.
\end{itemize}

\section{Weighted estimates}\label{SecCarleman}
In this section, we will prove several weighted estimates, including a weighted Poincar\'e inequality, in the goal of proving Theorems \ref{ThmCarleman} and \ref{ThmCarleman-t=0}.

Let $\Omega$ be a smooth bounded domain of $\mathbb R^n$, $n\geq 1$. In order to simplify the notations, we introduce the d'Alembertian operator:
$$
	\square  = \partial_t^2 - \Delta .
$$
In the following, we consider a function $v\in L^2(-T,T; H^1_0(\Omega))$ such that $\square v \in L^2(\Omega \times (-T,T))$ and $v=0$ on $\partial \Omega \times (-T,T)$.
We will first prove Carleman estimates similar to the ones of Theorems \ref{ThmCarleman}--\ref{ThmCarleman-t=0} for the operator $\square$, corresponding to a potential equal to zero.

\subsection{A Carleman estimate in arbitrary time $T$}
Let us now give a global (meaning ``up to the boundary'') Carleman inequality, following Imanuvilov's method \cite{Im02}. 

\begin{theorem}[see \cite{Baudouin01}]\label{CarlemanClas}
Assume the Gamma-condition \eqref{GCC-multiplier}. Let $\psi$ and $\varphi$ be the weight functions defined by  \eqref{poids}.

Then for every $\beta\in(0,1)$, there exist  $\lambda_0>0$, $s_{0}>0$ 
and a positive constant $M$
such that for all $s\geq s_{0}$ and $\lambda \geq \lambda_0$,
\begin{gather}
	s \lambda \int_{-T}^{T} \int_{\Omega}\varphi  e^{2s\varphi}(|\partial_t v|^2 + |\nabla v |^2)\,dxdt 
	+ s^3 \lambda^3 \int_{-T}^{T} \int_{\Omega} \varphi^3 e^{2s\varphi}|v|^2\,dxdt\nonumber\\
	+ \int_{-T}^{T} \int_{\Omega} |P_1 (e^{s\varphi} v)|^2 \,dxdt 
	\label{CarlemClas}
	\leq M\int_{-T}^{T} \int_{\Omega} e^{2s\varphi}| \square v|^2\,dxdt 
	+ Ms \int_{-T}^{T} \int_{\Gamma_{0}} \varphi e^{2s\varphi} \left|\partial_\nu v\right|^2\,d\sigma dt 
\end{gather}
for every $v\in L^2(-T,T;H_0^1(\Omega))$ satisfying $ \square v \in L^2(\Omega \times (-T,T))$, $\partial_\nu v \in L^2(\Gamma_0 \times (-T,T))$ and $v(\pm T) = \partial_tv(\pm T)=0$ in $\Omega$, and where $P_1$ is defined by 
\begin{equation}
	P_{1}w = \partial_t^2 w-\Delta w+s^2\lambda^2\varphi^2w(|\partial_t \psi|^2-|\nabla\psi|^2).\label{P1}
\end{equation}
\end{theorem}

Let us emphasize that such estimate is not new. Firstly, the proof can be read in the unpublished work \cite{Baudouin01} by the first author. Secondly, there are many Carleman estimates for hyperbolic equations.  One can find (local) Carleman estimates for regular functions with compact support in \cite{Carleman,HormanderLPDE,FursikovImanuvilov,Tataru96,Yam99}.
One can also read similar versions of global Carleman estimates for hyperbolic equations in \cite{Im02,ImYamIP03,Zhang00}.

For completeness, we give the proof of Theorem \ref{CarlemanClas} below.

\begin{remark}
This Carleman estimate is proved for any arbitrary time $T>0$, but $v$ has to satisfy $v(\pm T) = \partial_tv(\pm T)=0$. Therefore, the uniqueness result implied by Theorem~\ref{CarlemanClas} is not surprising since the corresponding unique continuation result is: If $v (\pm T) = \partial_t v (\pm T) = 0$, $v \in L^2((-T,T);H^1_0(\Omega))$, $\square v = 0$ and $\partial_\nu v_{|\Gamma_0} = 0$, then $v \equiv 0$.
\end{remark}

\begin{proof}
Using the weight $\varphi$ defined by \eqref{poids}, we set, for $s>0$,
$$
	w(x,t)=v(x,t)e^{s\varphi(x,t)} \quad \tn{for all } (x,t)\in \Omega \times (-T,T).
$$
Then, we introduce the conjugate operator $P$ defined by
\begin{equation}
\label{defP}
	Pw = e^{s\varphi}\square (e^{-s\varphi} w).
\end{equation}
Some easy computations give
\begin{eqnarray*}
Pw &=& \partial_t^2 w-2s\lambda\varphi (\partial_t w\partial_t \psi -\nabla w \cdot \nabla\psi)+s^2\lambda^2\varphi^2w(|\partial_t \psi|^2-|\nabla\psi|^2)
-\Delta w\\
&&\qquad  - s \lambda\varphi w(\partial_t^2 \psi - \Delta \psi )-s\lambda^2 \varphi w(|\partial_t \psi|^2-|\nabla\psi|^2)\\
&=& P_1w+P_2w+Rw,
\end{eqnarray*}
with $P_1$  defined by \eqref{P1} and
\begin{eqnarray}
		P_{2}w
		& = & (\alpha - 1)s \lambda\varphi w(\partial_t^2 \psi - \Delta \psi )
		-s\lambda^2 \varphi w(|\partial_t \psi|^2-|\nabla\psi|^2) \nonumber \\
		&&\qquad -2s\lambda\varphi (\partial_t w\partial_t \psi -\nabla w\cdot\nabla\psi)\label{P2} 
		\\
		Rw & = &-\alpha s \lambda\varphi w(\partial_t^2 \psi - \Delta \psi ),
\end{eqnarray}
where 
\begin{equation}
	\label{ConditionOnAlpha}
	\alpha \in \left( \dfrac{2\beta}{\beta + n},  \dfrac{2}{\beta + n} \right),
\end{equation}
 condition that will be explained later below in Step 2 of the proof.
Since we have
\begin{equation}
	\int_{-T}^{T} \int_{\Omega}\left(| P_{1}w| ^2 + | P_{2}w| ^2\right)\,dxdt
	+~2\int_{-T}^{T} \int_{\Omega}P_{1}wP_{2}w\,dxdt 
	= \int_{-T}^{T} \int_{\Omega}| Pw-Rw|^2\,dxdt\label{PP},
\end{equation}
the main part of the proof is then to bound from below the cross-term 
$$
	\int_{-T}^{T} \int_{\Omega}P_{1}w\, P_{2}w\,dxdt
$$ 
by positive and dominant terms, similar to the one of the left hand side of \eqref{CarlemClas}, and a negative boundary term, that will be moved to the right hand side of the estimate. For the sake of clarity, we will divide the proof in several steps.

All the computations below are done for smooth functions $v$ (equivalently, $w$). Then, by a classical density argument, we can extend the results to any $v$ satisfying $ \square v \in L^2(\Omega \times (-T,T))$, $\partial_\nu v \in L^2(\partial\Omega \times (-T,T))$ and $v(\pm T) = \partial_tv(\pm T)=0$ in $\Omega$.\\

\noindent \textbf{Step 1. Explicit calculations} 

We set 
$$
	\left\langle P_1w,P_2w\right\rangle_{L^2(\Omega\times(-T,T))}=\sum_{i,k=0}^3 I_{i,k}
$$ 
where $I_{i,k}$ is the integral of the product of the $i$th-term in $P_1w$ and the $k$th-term in $P_2w$. 
We mainly use integrations by parts and the properties of $w$ such as $w(\pm T) = 0$, $\partial_tw(\pm T) = 0$ in $\Omega$ and $w=0$ on $\partial\Omega\times (-T,T)$.\\

We shall also persistently use the fact that $\partial_t \psi$ does not depend on $x$ and that $\nabla \psi$ does not depend on time, and thus
$$
	\partial_t \nabla \psi  = \partial_{t}^2 \nabla \psi = {\bf 0}, \qquad \partial_t \Delta \psi = \partial_t (|\nabla \psi|^2)  = 0.
\medskip
$$

 Integrations by part in time give easily
\begin{eqnarray*}
	I_{11}&=&\int_{-T}^{T} \int_{\Omega}\partial_t^2 w((\alpha - 1)s \lambda\varphi w(\partial_t^2 \psi - \Delta \psi ))\,dxdt\\
	&=&(1-\alpha )s\lambda \int_{-T}^{T} \int_{\Omega} \varphi  | \partial_t w | ^2 (\partial_t^2 \psi - \Delta \psi ) \,dxdt\\
	&&- \dfrac {(1-\alpha )}2 s\lambda^2 \int_{-T}^{T} \int_{\Omega} \varphi | w | ^2 \partial_t^2 \psi 	(\partial_t^2 \psi - \Delta \psi ) \,dxdt\\
	&&- \dfrac {(1-\alpha )}2 s\lambda^3 \int_{-T}^{T} \int_{\Omega} \varphi | w | ^2  |\partial_t \psi|^2(\partial_t^2 \psi - \Delta \psi ) \,dxdt.
\end{eqnarray*}
Similarly, one has
\begin{eqnarray*}
	I_{12}&=&\int_{-T}^{T} \int_{\Omega}\partial_t^2 w(-s\lambda^2 \varphi w(|\partial_t \psi|^2-|\nabla\psi|^2))\,dxdt\\
	&=&s \lambda ^2 \int_{-T}^{T} \int_{\Omega} \varphi |\partial_t w|^2(|\partial_t \psi|^2-|\nabla\psi|^2)\,dxdt
	-s\lambda^2 \int_{-T}^{T} \int_{\Omega} \varphi  |w|^2 |\partial_t^2 \psi|^2\,dxdt\\
	&&-(2+\dfrac 12)s\lambda^3 \int_{-T}^{T} \int_{\Omega} \varphi |w|^2 |\partial_t \psi|^2\partial_t^2 \psi\,dxdt
	+\dfrac {s\lambda^3}2 \int_{-T}^{T} \int_{\Omega} \varphi |w|^2 |\nabla\psi|^2\partial_t^2 \psi\,dxdt\\
	&&-\dfrac {s\lambda^4}2 \int_{-T}^{T} \int_{\Omega} \varphi  |w|^2 |\partial_t \psi|^2 (|\partial_t \psi|^2-|\nabla\psi|^2)\,dxdt
\end{eqnarray*}
and
\begin{eqnarray*}
	I_{13}&=&\int_{-T}^{T} \int_{\Omega}\partial_t^2 w(-2s\lambda\varphi (\partial_t w\partial_t \psi -\nabla w\cdot\nabla\psi))\,dxdt\\
	&=& s\lambda \int_{-T}^{T} \int_{\Omega} \varphi | \partial_t w | ^2 \partial_t^2 \psi \,dxdt
	+s\lambda^2 \int_{-T}^{T} \int_{\Omega} \varphi | \partial_t w | ^2 |\partial_t \psi|^2 \,dxdt\\
	&&+ s\lambda \int_{-T}^{T} \int_{\Omega} \varphi | \partial_t w | ^2 \Delta \psi  \,dxdt
	+ s\lambda^2 \int_{-T}^{T} \int_{\Omega} \varphi | \partial_t w | ^2  |\nabla\psi|^2 \,dxdt\\
	&&-2s\lambda^2 \int_{-T}^{T} \int_{\Omega} \varphi \partial_t w\, \partial_t \psi \nabla w\cdot\nabla\psi \,dxdt.
\end{eqnarray*}

We compute in the same way
\begin{eqnarray*}
	I_{21}&=&\int_{-T}^{T} \int_{\Omega}-\Delta w ((\alpha - 1)s \lambda\varphi w(\partial_t^2 \psi - \Delta \psi ))\,dxdt\\
	&=& -(1-\alpha ) s \lambda \int_{-T}^{T} \int_{\Omega} \varphi |\nabla w|^2 (\partial_t^2 \psi - \Delta \psi ) \,dxdt\\
	&&+ \dfrac {(1-\alpha )}2 s \lambda^2 \int_{-T}^{T} \int_{\Omega} \varphi |w|^2 \Delta \psi (\partial_t^2 \psi - \Delta \psi )\,dxdt\\
	&&+ \dfrac {(1-\alpha )}2 s \lambda^3 \int_{-T}^{T} \int_{\Omega} \varphi |w|^2  |\nabla\psi|^2 (\partial_t^2 \psi - \Delta \psi )\,dxdt
\end{eqnarray*}
and 
\begin{eqnarray*}
	I_{22}&=&\int_{-T}^{T} \int_{\Omega}-\Delta w(-s\lambda^2\varphi w(|\partial_t \psi|^2-|\nabla\psi|^2))\,dxdt\\
	&=&-s \lambda^2 \int_{-T}^{T} \int_{\Omega} \varphi |\nabla w|^2(|\partial_t \psi|^2-|\nabla\psi|^2) \,dxdt\\
	&&-\dfrac{s \lambda^2}2 \int_{-T}^{T} \int_{\Omega} \varphi |w|^2 \Delta(|\nabla \psi|^2) \,dxdt\\
	&&+\dfrac {s \lambda^3}2 \int_{-T}^{T} \int_{\Omega} \varphi  |w|^2\Delta \psi (|\partial_t \psi|^2-|\nabla\psi|^2)\,dxdt\\
	&&+\dfrac {s \lambda^4}2 \int_{-T}^{T} \int_{\Omega} \varphi  |w|^2 |\nabla\psi|^2(|\partial_t \psi|^2-|\nabla\psi|^2)\,dxdt\\
	&&-s \lambda^3\int_{-T}^{T} \int_{\Omega} \varphi |w|^2 \nabla\psi \cdot \nabla(|\nabla\psi|^2) \,dxdt.
\end{eqnarray*}
Using the fact that $w|_{\partial\Omega\times (-T,T)}=0$, $\nabla w = (\partial_\nu w) \nu$ and $|\nabla w|^2 = |\partial_\nu w|^2$ on $\partial\Omega\times (-T,T)$: 
\begin{eqnarray*}
	I_{23}&=&\int_{-T}^{T} \int_{\Omega}-\Delta w(-2s\lambda\varphi (\partial_t w\partial_t \psi -\nabla w \cdot \nabla\psi))\,dxdt\\
	&=& s\lambda \int_{-T}^{T} \int_{\Omega} \varphi |\nabla w|^2 (\partial_t^2 \psi - \Delta \psi )\,dxdt
	+2s\lambda^2 \int_{-T}^{T} \int_{\Omega} \varphi | \nabla \psi\cdot\nabla w |^2 \,dxdt\\
	&&-2s\lambda^2 \int_{-T}^{T} \int_{\Omega} \varphi\partial_t w\, \partial_t \psi \nabla w\cdot\nabla\psi \,dxdt
	+s\lambda^2 \int_{-T}^{T} \int_{\Omega} \varphi | \nabla w |^2  (|\partial_t \psi|^2-|\nabla\psi|^2) \,dxdt\\
	&&- s\lambda \int_{-T}^{T} \int_{\partial\Omega} \varphi | \partial_\nu w|^2 \nabla \psi \cdot \nu \,d\sigma dt
	+4s\lambda \int_{-T}^{T} \int_{\Omega} \varphi  |\nabla w|^2\,dxdt
\end{eqnarray*}
since $D^2 \psi = 2 Id$.

One easily writes
\begin{eqnarray*}
	I_{31}&=&\int_{-T}^{T} \int_{\Omega}  s^2\lambda^2\varphi^2w(|\partial_t \psi|^2-|\nabla\psi|^2)
	((\alpha - 1)s \lambda\varphi w(\partial_t^2 \psi - \Delta \psi ))\,dxdt\\
	&=&(\alpha -1) s^3\lambda^3 \int_{-T}^{T} \int_{\Omega} \varphi^3 |w|^2(\partial_t^2 \psi - \Delta \psi ) (|\partial_t \psi|^2-|\nabla\psi|^2)\,dxdt
\end{eqnarray*}
and
\begin{eqnarray*}
	I_{32}&=&\int_{-T}^{T} \int_{\Omega}   s^2\lambda^2\varphi^2w(|\partial_t \psi|^2-|\nabla\psi|^2)
	(-s\lambda^2 \varphi w(|\partial_t \psi|^2-|\nabla\psi|^2))\,dxdt\\
	&=&-s^3\lambda^4 \int_{-T}^{T} \int_{\Omega}\varphi^3  |w|^2 (|\partial_t \psi|^2-|\nabla\psi|^2)^2\,dxdt.
\end{eqnarray*}
Finally, some integrations by part enable to obtain
\begin{eqnarray*}
	I_{33}&=&\int_{-T}^{T} \int_{\Omega}  s^2\lambda^2\varphi^2w(|\partial_t \psi|^2-|\nabla\psi|^2) 
	(-2s\lambda  \varphi (\partial_t w \partial_t \psi -\nabla w\cdot\nabla\psi)) \,dxdt\\
	&=&s^3\lambda ^3 \int_{-T}^{T} \int_{\Omega}\varphi^3 |w|^2 (\partial_t^2 \psi - \Delta \psi ) (|\partial_t \psi|^2-|\nabla\psi|^2) \,dxdt\\
	&&+2s^3\lambda ^3 \int_{-T}^{T} \int_{\Omega}\varphi^3  |w|^2 
		(\partial_t^2 \psi|\partial_t \psi|^2 + 2 |\nabla\psi|^2)\,dxdt\\
	&&+3s^3\lambda^4 \int_{-T}^{T} \int_{\Omega}\varphi^3  |w|^2  (|\partial_t \psi|^2-|\nabla\psi|^2)^2\,dxdt.
\end{eqnarray*}

Gathering all the terms that have been computed, we get 
\begin{equation}\label{P1P2}
	\begin{aligned}
	\int_{-T}^{T} \int_{\Omega} P_{1}&wP_{2}w\,dxdt\\
	=& ~2s\lambda \int_{-T}^{T} \int_{\Omega} \varphi | \partial_t w | ^2 \partial_t^2 \psi \,dxdt
	- \alpha s\lambda \int_{-T}^{T} \int_{\Omega} \varphi | \partial_t w | ^2 (\partial_t^2 \psi - \Delta \psi ) \,dxdt\\
	&+~2s\lambda^2 \int_{-T}^{T} \int_{\Omega}\varphi \left( | \partial_t w | ^2 |\partial_t \psi|^2
	-2\partial_t w\, \partial_t \psi  \nabla w\cdot \nabla\psi + | \nabla \phi\cdot\nabla w |^2 \right) \,dxdt\\
	&+~4s\lambda \int_{-T}^{T} \int_{\Omega} \varphi   |\nabla w|^2\,dxdt
	+\alpha s \lambda \int_{-T}^{T} \int_{\Omega} \varphi |\nabla w|^2 (\partial_t^2 \psi - \Delta \psi ) \,dxdt\\
	&-~s\lambda \int_{-T}^{T} \int_{\partial\Omega}\varphi \left| \partial_\nu w\right |^2 \nabla\psi\cdot\nu(x)\,d\sigma dt\\
	&+~2s^3\lambda^4 \int_{-T}^{T} \int_{\Omega} \varphi^3 |w|^2 (|\partial_t \psi|^2-|\nabla\psi|^2)^2\,dxdt\\
	&+~2 s^3\lambda ^3 \int_{-T}^{T} \int_{\Omega}\varphi^3  |w|^2
		(\partial_t^2 \psi|\partial_t \psi|^2 + 2 |\nabla\psi|^2 )\,dxdt\\
	&+~\alpha s^3\lambda^3 \int_{-T}^{T} \int_{\Omega}\varphi^3  |w|^2
		(\partial_t^2 \psi - \Delta \psi ) (|\partial_t \psi|^2-|\nabla\psi|^2)\,dxdt+~ X_1,
	\end{aligned}
\end{equation}
where $X_1$ gathers the non-dominating terms and satisfies
$$
|X_1| \leq M s\lambda^4 \int_{-T}^{T} \int_{\Omega} \varphi |w|^2  \,dxdt.
$$
Note that, since $\psi \geq 1$, we have, for $\lambda$ large enough, $\lambda^4 \leq e^{2\lambda \psi} = \varphi^2$, and therefore, 
\begin{equation}
	\label{X1}
	|X_1| \leq M s \int_{-T}^T \int_\Omega \varphi^3 |w|^2 dx dt
\end{equation}
for some $M$ independent of $s$ and $\lambda$.
Here and in the rest of the proof of Theorem \ref{CarlemanClas}, $M>0$ corresponds to a generic constant depending at least on $\Omega$ and $T$ but independent of $s$ and $\lambda$.\\

\noindent \textbf{Step 2. Dominating terms} 

On the one hand, about the first order derivative terms, one can notice that 
\begin{multline}\label{00}
	2s\lambda^2 \int_{-T}^{T} \int_{\Omega}\varphi \left( | \partial_t w | ^2 |\partial_t \psi|^2
		-2\partial_t w \partial_t \psi  \nabla w \cdot \nabla\psi + | \nabla \psi\cdot\nabla w |^2 \right)\,dxdt\\
	=2s\lambda^2 \int_{-T}^{T} \int_{\Omega}\varphi \left( \partial_t w \partial_t \psi- \nabla w \cdot \nabla\psi\right)^2 \,dxdt \geq 0.
\end{multline}
Since this term can vanish, we focus on the terms in $s\lambda$ in $| \partial_t w|^2 $ and $|\nabla w|^2$ and we want them to be strictly positive, what is equivalent to having:
$$
	2 \partial_t^2 \psi -\alpha (\partial_t^2 \psi - \Delta\psi) > 0 \quad \tn{ and } \quad
	4 +\alpha  (\partial_t^2 \psi - \Delta \psi ) > 0.
$$
By explicit computations, these two terms are positive if and only if $\alpha$ satisfies \eqref{ConditionOnAlpha}. This justifies the choice of the parameter $\alpha$. Note also that this can be satisfied only if $\beta\in(0,1)$.

As a direct consequence, we can write 
\begin{multline}
	\label{ordre1}
	2 s\lambda \int_{-T}^{T} \int_{\Omega} \varphi | \partial_t w|^2\partial_t^2 \psi \,dxdt 
	-\alpha s\lambda \int_{-T}^{T} \int_{\Omega} \varphi | \partial_t w|^2(\partial_t^2 \psi - \Delta\psi) \,dxdt\\
	+4s\lambda \int_{-T}^{T} \int_{\Omega} \varphi  |\nabla w|^2 \,dxdt
	+\alpha s \lambda \int_{-T}^{T} \int_{\Omega} \varphi |\nabla w|^2 (\partial_t^2 \psi - \Delta \psi ) \,dxdt\\
	\geq M s\lambda \int_{-T}^{T} \int_{\Omega} \varphi | \partial_t w|^2\,dxdt 
	+ M s\lambda \int_{-T}^{T} \int_{\Omega} \varphi | \nabla w|^2 \,dxdt.
\end{multline}

On the other hand, considering the $0$-th order terms, we can observe that:
\begin{multline*}
	2s^3\lambda^4 \int_{-T}^{T} \int_{\Omega}\varphi^3 |w|^2  (|\partial_t \psi|^2-|\nabla\psi|^2)^2\,dxdt\\
	+~\alpha s^3\lambda^3 \int_{-T}^{T} \int_{\Omega}\varphi^3 |w|^2(\partial_t^2 \psi - \Delta \psi )
		 (|\partial_t \psi|^2-|\nabla\psi|^2)\,dxdt\\
		 +~2s^3\lambda ^3 \int_{-T}^{T} \int_{\Omega}\varphi^3 |w^2| 
		(\partial_t^2 \psi|\partial_t \psi|^2 + 2|\nabla\psi|^2)\,dxdt\\
	= s^3\lambda ^3 \int_{-T}^{T} \int_{\Omega}\varphi^3 |w|^2  F_{\lambda}(\phi) \,dxdt,
\end{multline*}
where 
\begin{align*}
	F_{\lambda}(\phi)&= 
	2\lambda(|\partial_t \psi|^2-|\nabla\psi|^2)^2 +  2(\partial_t^2 \psi|\partial_t \psi|^2 + 2 |\nabla\psi|^2 )
		+ \alpha (\partial_t^2 \psi - \Delta \psi ) (|\partial_t \psi|^2-|\nabla\psi|^2)\\
	&= 2\lambda(|\partial_t \psi|^2-|\nabla\psi|^2)^2 +  (2\partial_t^2 \psi + \alpha (\partial_t^2 \psi - \Delta \psi ))(|\partial_t \psi|^2 - |\nabla\psi|^2 )
		+ 2(\partial_t^2 \psi + 2) |\nabla\psi|^2\\
	&= 2 \lambda X^2+  (2\partial_t^2 \psi + \alpha (\partial_t^2 \psi - \Delta \psi )) X  +16(1-\beta)|x-x_0|^2.
\end{align*}
with $X = |\partial_t \psi|^2-|\nabla\psi|^2$.

Since $x_0\not\in\overline \Omega$ and $\beta\in(0,1)$, we have $16(1-\beta)|x-x_0|^2\geq c^*>0$. 
Therefore, we are considering a polynomial $A(X) \geq 2\lambda X^2 - 2\left( \alpha (\beta +n) + 2 \beta \right) X+c^*$ and taking $\lambda>0$ large enough, the minimum of $A$ will be strictly positive.
Consequently, 
\begin{equation}
	\label{ordre0}
	s^3\lambda ^3 \int_{-T}^{T} \int_{\Omega} \varphi^3|w^2| F_{\lambda}(\phi) \,dxdt\\
	\geq M s^3\lambda^3 \int_{-T}^{T} \int_{\Omega} \varphi^3 | w|^2\,dxdt.
\end{equation}\\

Thus, plugging \eqref{00}, \eqref{ordre1} and \eqref{ordre0} in \eqref{P1P2}, and since $\nabla \psi = 2(x-x_0)$, we  obtain
\begin{multline*}
	\int_{-T}^{T} \int_{\Omega}P_{1}wP_{2}w\,dxdt  
	+ 2 s \lambda\int_{-T}^{T} 	\int_{\partial\Omega}\varphi \left| \partial_\nu w\right|^2 (x-x_0)\cdot\nu(x)\,d\sigma dt + |X_1| \\
	\geq M s\lambda \int_{-T}^{T} \int_{\Omega} \varphi\left(| \partial_t w| ^2+|\nabla w|^2\right)\,dxdt
		+ M s^3\lambda ^3\int_{-T}^{T} \int_{\Omega}\varphi^3 | w| ^2\,dxdt.
\end{multline*}
Since we also have
\begin{eqnarray*}
	\int_{-T}^{T} \int_{\Omega}| Pw-Rw|^2\,dxdt 
		&\leq &
		2\int_{-T}^{T} \int_{\Omega}| Pw|^2\,dxdt 
		+2\int_{-T}^{T} \int_{\Omega}| Rw|^2\,dxdt
		\\
		&\leq &
		M\int_{-T}^{T} \int_{\Omega}| Pw|^2\,dxdt 
		+Ms^2 \lambda^2\int_{-T}^{T} \int_{\Omega} \varphi^2 |w|^2 \,dxdt,
\end{eqnarray*}
using \eqref{PP} and \eqref{X1}, we get 
\begin{gather*}
	s\lambda \int_{-T}^{T} \int_{\Omega}\left(| \partial_t w| ^2+|\nabla w|^2\right) \varphi\,dxdt
		+ s^3\lambda ^3\int_{-T}^{T} \int_{\Omega} | w| ^2\varphi^3\,dxdt\\
	+\int_{-T}^{T} \int_{\Omega}\left(| P_{1}w| ^2 + | P_{2}w| ^2\right)\,dxdt\\
	\leq M\int_{-T}^{T} \int_{\Omega}| Pw|^2\,dxdt 
		+M s \lambda\int_{-T}^{T} \int_{\Gamma_0}\varphi \left| \partial_\nu w\right|^2 
			(x-x_0)\cdot\nu(x)\,d\sigma dt \\
	+~M s \int_{-T}^{T} \int_{\Omega}\varphi^3 | w| ^2\,dxdt  + Ms^2 \lambda^2\int_{-T}^{T} \int_{\Omega} \varphi^2 |w|^2 \,dxdt.
\end{gather*}
We take now $s_0$ large enough so that the terms of the last line (coming from  $X_1$ and $| Rw|^2$) are absorbed by the dominant term in $s^3\lambda^3  |w|^2\varphi^3$ as soon as $s\geq s_0$. Using also the condition \eqref{GCC-multiplier} on $\Gamma_0$, we finally obtain for some positive constant $M$,
\begin{gather}
	s\lambda\int_{-T}^{T} \int_{\Omega} \varphi(| \partial_t w|^2+|\nabla w|^2)\,dxdt
	+ s^3\lambda ^3\int_{-T}^{T} \int_{\Omega} \varphi^3 |w|^2\,dxdt\nonumber\\
	+\int_{-T}^{T} \int_{\Omega}|P_1 w|^2\,dxdt+\int_{-T}^{T} \int_{\Omega}|P_2 w|^2\,dxdt\label{CarlemW}\\
	\leq M\int_{-T}^{T} \int_{\Omega}| Pw|^2\,dxdt 
	+Ms\lambda \int_{-T}^{T} \int_{\Gamma_0}  \varphi\left|\partial_\nu w\right|^2\,d\sigma dt \nonumber
\end{gather}
for all $ s\geq s_0$ and $ \lambda \geq  \lambda_0$.\\

\noindent \textbf{Step 3. Back to the variable $v$} 

Since $w=ve^{s\varphi}$, we have
\begin{align*}
	e^{2s\varphi}|\partial_t v|^2 &\leq 2|\partial_t w|^2+2s^2|\partial_t\varphi|^2|w|^2\leq 2|\partial_t w|^2+Cs^2\lambda^2 \varphi^2 |w|^2, && \tn{ in } \Omega \times (-T,T),\\
	e^{2s\varphi}|\nabla v|^2 &\leq 2|\nabla w|^2+2s^2|\nabla\varphi|^2|w|^2\leq  2|\nabla w|^2+Cs^2\lambda^2 \varphi^2 |w|^2,&& \tn{ in } \Omega \times (-T,T),\\
	e^{2s\varphi}\left|\partial_\nu v\right|^2&= \left|\partial_\nu w\right|^2, && \tn{ on }\partial\Omega \times (-T,T).
\end{align*}
Using \eqref{defP} that gives by construction $P w = e^{s\varphi} \square v$, 
we can go back to the variable $v$ in \eqref{CarlemW} and obtain that there exists some positive constant $M$ such that for all $ s\geq s_0$ and $ \lambda \geq \lambda_0$,
\begin{gather*}
	s\lambda \int_{-T}^{T} \int_{\Omega} \varphi e^{2s\varphi}(|\partial_t v|^2 + |\nabla v |^2)\,dxdt 
	+ s^3\lambda^3\int_{-T}^{T} \int_{\Omega} \varphi^3 e^{2s\varphi}|v|^2\,dxdt\\
	+ \int_{-T}^{T} \int_{\Omega} |P_1 (e^{s\varphi} v)|^2 \,dxdt + \int_{-T}^{T} \int_{\Omega} |P_2 (e^{s\varphi} v)|^2\,dxdt\\
	\leq M\int_{-T}^{T} \int_{\Omega} e^{2s\varphi}| \square v|^2\,dxdt 
	+ Ms\lambda \int_{-T}^{T} \int_{\Gamma_{0}} \varphi e^{2s\varphi} \left|\partial_\nu v\right|^2\,d\sigma dt.
\end{gather*}
This concludes the proof of Theorem \ref{CarlemanClas}.
\end{proof}

In the sequel, we will fix $\lambda = \lambda_0$ and use the fact that $\varphi$ then is bounded from below by~$1$ and from above by some constants depending on $\lambda$. Since $\lambda$ is fixed, we can put it into the constants and obtain the following result: 

\begin{corollary}\label{CorCarlemanClas}
Assume the Gamma-condition \eqref{GCC-multiplier}.

Then for every $\beta\in(0,1)$, there exist  $\lambda>0$, $s_{0}>0$ 
and a positive constant $M$
 such that for all $s\geq s_{0}$,
\begin{multline}
	s \int_{-T}^{T} \int_{\Omega} e^{2s\varphi}(|\partial_t v|^2 + |\nabla v |^2)\,dxdt 
	+ s^3\int_{-T}^{T} \int_{\Omega} e^{2s\varphi}|v|^2\,dxdt
	+ \int_{-T}^{T} \int_{\Omega} |P_1 (e^{s\varphi} v)|^2 \,dxdt 
	\\
	\label{CarlemClas-sansLambda}
	\leq M\int_{-T}^{T} \int_{\Omega} e^{2s\varphi}| \square v|^2\,dxdt 
	+ Ms \int_{-T}^{T} \int_{\Gamma_{0}} e^{2s\varphi} \left|\partial_\nu v\right|^2\,d\sigma dt 
\end{multline}
for every $v\in L^2((-T,T);H_0^1(\Omega))$ satisfying $ \square v \in L^2(\Omega \times (-T,T))$, $\partial_\nu v \in L^2(\partial\Omega \times (-T,T))$ and $v(\pm T) = \partial_tv(\pm T)=0$ in $\Omega$, where $P_1$ is defined in \eqref{P1}.
\end{corollary}

\noindent In the following, $M$ denotes various constants that do not depend on the parameter $s$.

\subsection{Weighted Poincar\'e inequality}
We prove here a weighted version of the Poincar\'e inequality that will be used thereafter.

\begin{lemma}\label{Poincare}
Let $\tilde \varphi \in C^2 (\overline{\Omega})$ and assume that the weight $\tilde \varphi$ defined on $\overline{\Omega}$ satisfies
\begin{equation}\label{hyp1}
\inf_\Omega |\nabla \tilde \varphi| \geq \delta >0.
\end{equation}
Then there exist $s_0>0$ and $M>0$ such that, for all $s \geq s_0$ and for all $z \in  H_0^1(\Omega)$,
\begin{equation}
	\label{PoincareEst}
	s^2  \int_{\Omega} e^{2s\tilde \varphi} |z|^2 \,dx
	\leq M \int_{\Omega}e^{2s\tilde \varphi} |\nabla z|^2 \,dx.
\end{equation}
\end{lemma}

\begin{proof} We begin with the following computation, that uses the identity
$\nabla\cdot (e^{2s\tilde \varphi} \nabla \tilde \varphi ) = e^{2s\tilde \varphi} \Delta\tilde \varphi + 2s e^{2s\tilde \varphi} |\nabla\tilde \varphi|^2$:
\begin{eqnarray}
	s^2  \int_{\Omega} e^{2s\tilde \varphi} |z|^2 |\nabla \tilde \varphi|^2 \,dx
	&= &
	\frac{s}{2}  \int_{\Omega}  |z|^2 \left( \nabla \cdot \left(e^{2s\tilde \varphi} \nabla \tilde \varphi \right) - e^{2s\tilde \varphi} \Delta \tilde \varphi \right) \,dx 
	\nonumber\\
	& = &
	 - s  \int_{\Omega} e^{2s\tilde \varphi} z   \nabla z \cdot \nabla \tilde \varphi  \,dx - \frac{s}{2} \int_{\Omega} e^{2s\tilde \varphi} |z|^2  \Delta \tilde \varphi  \,dx.\label{eq11}
\end{eqnarray}
Since $\tilde \varphi \in C^2 (\overline{\Omega})$ and since we suppose \eqref{hyp1}, the last term in \eqref{eq11} can be bounded as follows:
$$
	- \frac{s}{2} \int_{\Omega} e^{2s\tilde \varphi}  |z|^2  \Delta \tilde \varphi \,dx \leq M s  \int_{\Omega} e^{2s\tilde \varphi} |z|^2 |\nabla \tilde \varphi|^2 \,dx.
$$
Then, for $s$ sufficiently large, this term is absorbed in the left hand side of (\ref{eq11}). This implies the following estimate:
\begin{align*}
	s^2  \int_{\Omega} e^{2s\tilde \varphi} |z|^2 |\nabla \tilde \varphi|^2 \,dx
	& \leq 
	M s  \int_{\Omega} \left| e^{2s\tilde \varphi} z \nabla \tilde \varphi \right| \left| e^{2s\tilde \varphi}\nabla z \right| \,dx 
	\\
	& \leq
	M \left( s^2 \int_{\Omega} e^{2s\tilde \varphi} |z|^2 |\nabla \tilde \varphi |^2 \,dx \right)^{1/2} \left( \int_{\Omega} e^{2s\tilde \varphi} |\nabla z|^2 \,dx \right)^{1/2}.
\end{align*}
This yields \eqref{PoincareEst} and concludes the proof of Lemma \ref{Poincare}.
\end{proof}

Let us emphasize that the weight function $\varphi$ defined by \eqref{poids} is such that for all $t \in (-T,T)$, $\varphi(\cdot, t)$ satisfies assumption \eqref{hyp1} of Lemma~\ref{Poincare} since $x_0\not\in\overline\Omega$ and $\nabla\varphi = 2\lambda(x-x_0)\varphi $.

\subsection{A Carleman estimate in time $T$ large enough}
When the time $T$ is large enough in the sense of \eqref{GCC-Time}, we claim that the conditions at times $\pm T$ can be removed of the assumptions of Theorem~\ref{CarlemanClas}. Roughly speaking, this will follow from an energy argument coupled to the Carleman estimate \eqref{CarlemClas}. The result is given in the following theorem:

\begin{theorem}\label{Carlemangral}
Assume the multiplier condition \eqref{GCC-multiplier} and the time condition  \eqref{GCC-Time}. Define the weight functions $\varphi$ as in \eqref{poids} with $\beta \in (0,1)$ being such that \eqref{GCC-Time-Beta} holds.

Then there exist 
$s_{0}>0$ and a positive constant $M $
such that for all $ s \geq s_{0}$:
\begin{equation}
	\begin{aligned}\label{Carlemgral}
		s \int_{-T}^{T} \int_{\Omega} &e^{2s\varphi}\left(|\partial_t z|^2 + |\nabla z |^2 \right)\,dxdt 
		+ s^3\int_{-T}^{T} \int_{\Omega} e^{2s\varphi}|z|^2\,dxdt\\
		&\leq M\int_{-T}^{T} \int_{\Omega} e^{2s\varphi}| \square z|^2\,dxdt 
		+ Ms \int_{-T}^{T} \int_{\Gamma_{0}} e^{2s\varphi} \left|\partial_\nu z\right|^2\,d\sigma dt,
	\end{aligned}
\end{equation}
for all $z\in L^2((-T,T);H_0^1(\Omega))$ satisfying $\square z \in L^2(\Omega\times (-T,T))$ and 
$\partial_\nu z \in L^2(\partial\Omega\times (-T,T))$.
\end{theorem}

\begin{remark}
Let us emphasize that Theorem \ref{Carlemangral}, contrary to Corollary \ref{CorCarlemanClas}, does not require $z(\pm T) = \partial_t z(\pm T) = 0$ in $\Omega$. 
\end{remark}

\begin{proof}
From condition \eqref{GCC-Time-Beta} that states $	\sup_{x\in \Omega}|x-x_0| < \beta T$, 
we can choose $\eta \in (0,T)$ and $\varepsilon \in (0,1)$  such that
\begin{equation}
	\label{ConditionsT-Eta-Beta}
	(1- \varepsilon)(T- \eta) \beta  \geq \sup_{x \in \Omega} |x- x_0|.
\end{equation}
Then, explicit computations on $\psi(x,t) = |x-x_0|^2-\beta t^2 + C_0$ show that we have
\begin{equation}
\label{EstimatePsi-Eta}
	\forall t \in (-T,-T+\eta)\cup (T-\eta,T), \qquad 
	\left\{ 
		\begin{array}{l}
		\ds (1- \varepsilon)|\partial_t \psi(t) | \geq \sup_{x \in \Omega} |\nabla \psi(x,t)|, 
			\\
		\ds \sup_{x \in \Omega}\psi(x,t) < C_0 < \inf_{x \in \Omega} \psi(x,0).
		\end{array}
	\right.
\end{equation}

Hence, we introduce the cut-off function $\chi \in C^{\infty}_c(\mathbb{R})$ such that $0\leq \chi\leq 1$ and
\begin{equation}\label{chi}
	\chi(t) = \left\{\begin{array}{lll} 1, & \text{if} & -T+\eta \leq t \leq T-\eta, \\ 
	0, & \text{if}  &t \leq -T \quad \text{or} \quad t \geq T, \end{array}\right.
\end{equation}
and we set $v = \chi z$, in $\Omega \times (-T,T)$. Therefore, $v$ satisfies the required hypothesis $v(\pm T)= \partial_t v(\pm T) = 0$ in $\Omega$ and we can apply the Carleman estimate~\eqref{CarlemClas-sansLambda} of Corollary \ref{CorCarlemanClas} to $v$:
\begin{equation}
	\begin{aligned}\label{eq2}	
		s\int_{-T}^T \int_\Omega &e^{2s\varphi}\left( |\partial_t v|^2 + |\nabla v|^2 \right)\,dxdt
		+s^3\int_{-T}^T \int_\Omega e^{2s\varphi} |v|^2\,dxdt\\ 
		&\leq M \int_{-T}^T \int_\Omega e^{2s\varphi} |\square  v|^2 \,dxdt
		+M s\int_{-T}^T \int_{\Gamma_0} e^{2s\varphi} |\partial_\nu v|^2\,d\sigma dt.
	\end{aligned}
\end{equation}
One can calculate that
$$
	\square  v =\chi \square  z-2\chi'\partial_t z -\chi'' z,\qquad   \tn{in } \Omega \times (-T,T).
$$
Besides, the functions $\chi'$ and $\chi''$ have compact support in $(-T,-T+\eta)\cup (T-\eta,T)$. Thus, from (\ref{eq2}), we deduce the following estimate on $z$:
\begin{multline}\label{eq3}
		s\int_{-T+\eta}^{T-\eta} \int_\Omega e^{2s\varphi}\left( |\partial_t z|^2 + |\nabla z|^2 \right)dxdt		+s^3\int_{-T+\eta}^{T-\eta} \int_\Omega e^{2s\varphi} |z|^2\,dxdt\\
		 \leq M \int_{-T}^T \int_\Omega e^{2s\varphi} |\square  z|^2 \,dxdt		
		+M s\int_{-T}^T \int_{\Gamma_0} e^{2s\varphi} |\partial_\nu z|^2\,dxdt \\
		+M \int_{T-\eta}^T \int_{\Omega}  e^{2s\varphi}\left( |\partial_t z|^2 + |z|^2 \right)d\sigma dt 
		+M \int_{-T}^{-T+\eta} \int_{\Omega}  e^{2s\varphi}\left( |\partial_t z|^2 + |z|^2 \right)dxdt.
\end{multline}
We will show that the last two terms of (\ref{eq3}) can be absorbed in the left hand side if the parameter $s$ and the time $T$ are chosen sufficiently large. In order to do that, we introduce the following weighted energy:

$$
	E_s(t) = \frac{1}{2}\int_\Omega e^{2s\varphi(t)}\left( |\partial_t z(t)|^2 + |\nabla z(t)|^2  \right)dx. \qquad \forall t \in (-T,T),
$$
We first calculate:
$$
		\frac{dE_s}{dt} = s \int_\Omega e^{2s\varphi}  \partial_t \varphi  \left( |\partial_t z|^2 + |\nabla z|^2 \right)dx+\int_\Omega e^{2s\varphi}\left( \partial_t z \partial^2_t z + \nabla z\cdot \nabla (\partial_t z)  \right)dx.
$$
Thus, after an integration by parts,
\begin{equation}\label{violet}
	\begin{aligned}
		\frac{dE_s}{dt} - s\int_\Omega  e^{2s\varphi} \partial_t \varphi \left( |\partial_t z|^2 + |\nabla z|^2  \right)dx + 2 s\int_\Omega e^{2s\varphi}  \partial_t z   \nabla \varphi \cdot  \nabla z \,dx
		=  \int_\Omega e^{2s\varphi} \partial_t z \square z \,dx.
	\end{aligned}
\end{equation}

\noindent \textbf{Step 1: Term of \eqref{eq3} on $(T- \eta, T)$}

{\bf (a)} Thanks to the formula $2ab \leq a^2+b^2$, we can bound by below the left hand side of equality \eqref{violet} leading to:
$$
	\frac{dE_s}{dt}- s\int_\Omega  e^{2s\varphi} \left( \partial_t \varphi +  |\nabla \varphi|\right)  
	\left( |\partial_t z|^2 + |\nabla z|^2 \right)dx 
	\leq \int_\Omega e^{2s\varphi} \partial_t z \square z \, dx.
$$
According to \eqref{EstimatePsi-Eta} and \eqref{poids}, for $t \in (T- \eta, T)$
$$
	\inf_{\Omega}\left\{-(\partial_t \varphi + |\nabla \varphi|)\right\} 
	\geq \inf_{\Omega}\left\{ -\varepsilon \partial_t \varphi\right\} 
	\geq 2 \varepsilon \beta (T- \eta) e^{\lambda \psi} \geq 2 \varepsilon \beta (T- \eta) : = c_* >0.
$$ 
Thus, 
\begin{equation}\label{red}
	\forall t\in (T-\eta,T), \quad \frac{dE_s}{dt}+ sc_* \int_\Omega  e^{2s\varphi}\left( |\partial_t z|^2 + |\nabla z|^2\right)dx
	\leq \int_\Omega e^{2s\varphi} \partial_t z \square z \, dx.
\end{equation}
Now, using the formula $2ab \leq \epsilon a^2+\dfrac{b^2}{\epsilon}$ with $\epsilon = sc_*$, we can bound the right hand side of (\ref{red}) as follows:
$$
	\left| \int_\Omega e^{2s\varphi} \partial_t z  \square z\,dx \right| 
	\leq \frac{sc_*}{2}   \int_\Omega  e^{2s\varphi} |\partial_t z|^2\,dx 
	+ \frac{1}{2sc_* }  \int_\Omega e^{2s\varphi} |\square  z|^2\,dx.
$$
The first term of the right hand side is absorbed by the left hand side of (\ref{red}). We obtain
$$
	\frac{dE_s}{dt}+ \frac{ sc_*}{2}  \int_\Omega  e^{2s\varphi}\left( |\partial_t z|^2 + |\nabla z|^2\right)dx
	\leq \frac{1}{2sc_*} \int_\Omega e^{2s\varphi}  |\square z|^2\,dx
$$
or, equivalently,
$$
	\frac{dE_s}{dt}+ s c_* E_s  \leq \frac{1}{2sc_*} \int_\Omega e^{2s\varphi}  |\square z|^2\,dx.
$$
Using the Gr\" onwall lemma, we can write, for all $t \in (T-\eta,T)$,
\begin{eqnarray}
		E_s(t)
			& \leq &
		E_s(T-\eta)e^{s c_* (T-\eta-t)}  
			+ 
		\frac{1}{2sc_*} \int_{T-\eta}^t e^{sc_* (\tau- t)} \int_\Omega e^{2s\varphi(\tau)}  |\square z(\tau)|^2 \,dx d\tau \nonumber
		\\
			&\leq &
		E_s(T- \eta) e^{-sc_*(t- (T-\eta))} 
		+
		\frac{1}{2sc_*} \int_{T- \eta}^T  \int_\Omega e^{2s\varphi(\tau)}  |\square  z(\tau) |^2 \,dx d \tau.\label{rose1}
\end{eqnarray}
Integrating this relation for $t$ between $T-\eta$ and $T$, we obtain:
\begin{eqnarray}
		\int_{T-\eta}^T E_s(t) dt &\leq& E_s(T- \eta) \int_{T-\eta}^T e^{-sc_*(t- (T-\eta))} dt 
		+\frac{\eta}{2sc_*} \int_{-T}^T  \int_\Omega e^{2s\varphi}  |\square  z |^2 \,dx dt
		\nonumber \\
		&\leq& \frac{M}{s} E_s(T- \eta) 
		+ \frac{M}{s} \int_{-T}^T  \int_\Omega e^{2s\varphi}  |\square  z |^2 \,dxdt. \label{EstIntegreeT-eta}
\end{eqnarray}

{\bf (b)} Now we want to estimate $E_s(T- \eta)$ by $E_s(\tau)$ for $\tau \in (-T+\eta,T-\eta)$
. We use equality \eqref{violet} that we integrate between $\tau$ and $T-\eta$:
$$
	\begin{aligned}
		E_s(T-\eta) - &E_s(\tau) 
		= s\int_\tau^{T-\eta }\int_\Omega e^{2s\varphi}  \partial_t \varphi \left( |\partial_t z|^2 + |\nabla z|^2  \right)dxdt \\
		&-2s \int_\tau^{T-\eta }\int_\Omega e^{2s\varphi}  \partial_t z   \nabla \varphi \cdot  \nabla z \,dxdt
		+  \int_\tau^{T-\eta }\int_\Omega e^{2s\varphi} \partial_t z\square z dxdt.
	\end{aligned}
$$
We bound the right hand side using Cauchy-Schwarz (and bounds on $\varphi$ and its derivatives): 
$$
	E_s(T-\eta) - E_s(\tau) \leq M s \int_{-T+\eta}^{T-\eta }E_s(t)dt + \frac{M}{s}\int_{-T}^{T}\int_\Omega e^{2s\varphi} | \square z |^2\,dxdt.
$$
Integrating for $\tau$ between  $-T+\eta$ and $T-\eta$, since $s$ is large,
\begin{equation}\label{orange1} 
	E_s(T-\eta) \leq  M s \int_{-T+\eta}^{T-\eta }E_s(t)dt
	+ \frac{M}{s}\int_{-T}^{T}\int_\Omega e^{2s\varphi} | \square z |^2\,dxdt.
\end{equation}

{\bf (c)} Finally, thanks to \eqref{EstIntegreeT-eta} and \eqref{orange1}, we deduce
\begin{multline}\label{firstterm1}
	\int_{T-\eta}^T \int_{\Omega}  e^{2s\varphi}\left( |\partial_t z|^2 + | \nabla z|^2 \right) dxdt\\
	\leq  
	M\int_{-T+\eta}^{T-\eta}\int_{\Omega} e^{2s\varphi}\left(|\partial_t z|^2 + |\nabla z |^2 \right)dxdt
	+ 
	\frac{M}{s}\int_{-T}^T \int_\Omega e^{2s\varphi} |\square  z|^2  \,dxdt	.
\end{multline}
Moreover, combining it with the weighted Poincar\'e estimate of Lemma \ref{Poincare}, we get
\begin{multline}
	\label{EstimeeEnergieT}
	s \int_{T- \eta}^T \int_\Omega e^{2s\varphi}\left(|\partial_t z|^2 + |\nabla z |^2  + s^2|z|^2\right)dxdt
	\leq 
	M s \int_{T-\eta}^T E_s(t) dt 
	\\
	\leq M s \int_{-T+\eta}^{T-\eta }E_s(t)dt
	+ M \int_{-T}^{T}\int_\Omega e^{2s\varphi} | \square z |^2\,dxdt.
\end{multline}

\noindent \textbf{Step 2: Term of \eqref{eq3}  on $(-T,-T+ \eta)$}

We want to obtain the same results as previously but on the interval $(-T,-T+ \eta)$. In order to do that, one can introduce $\tilde z(x,t) = z(x, -t)$ and apply the above estimates to $\tilde z$ (we make the change of variable $t\to -t$). Thus, equations \eqref{rose1}--\eqref{orange1} coincide with the following ones:
\begin{multline}
	\label{rose} 
	\forall t \in (-T, -T + \eta),
	\quad 
	E_s(t) ~ \leq ~  E_s(-T+ \eta) e^{-s c_*(-T+\eta-t)} 
	\\+ \frac{C}{s} \int_{-T}^{-T+\eta}  \int_\Omega e^{2s\varphi(\tau)}  |\square z(\tau) |^2 \, dxd\tau,
\end{multline}
\begin{eqnarray} 
	\label{EstIntegree-T+eta}
	\int_{-T}^{-T+\eta} E_s(t) dt 
	&\leq & \frac{M}{s} E_s(-T+ \eta) 
		+ \frac{M}{s} \int_{-T}^T  \int_\Omega e^{2s\varphi}  |\square  z |^2 \,dxdt,
	\\ 
	\label{orange} 
	 E_s(-T+\eta)
	 &\leq& M s \int_{-T+\eta}^{T-\eta }E_s(t)dt
	+ \frac{M}{s}\int_{-T}^{T}\int_\Omega e^{2s\varphi} | \square z |^2\,dxdt.
\end{eqnarray}
Combining \eqref{EstIntegree-T+eta} with \eqref{orange}, we deduce
\begin{multline}\label{firstterm}
	\int_{-T}^{-T+\eta} \int_{\Omega}  e^{2s\varphi}\left( |\partial_t z|^2 + | \nabla z|^2 \right)dxdt \\
		\leq M \int_{-T+\eta}^{T-\eta}\int_{\Omega} e^{2s\varphi}\left(|\partial_t z|^2 + |\nabla z |^2 \right)dxdt + \frac{M}{s}\int_{-T}^T \int_\Omega e^{2s\varphi} |\square z|^2\,dxdt .
\end{multline}
Besides, similarly to \eqref{EstimeeEnergieT}, using Lemma \ref{Poincare}, we have
\begin{multline}
	\label{EstimeeEnergie-T}
	s \int_{-T}^{-T+\eta} \int_\Omega e^{2s\varphi}\left(|\partial_t z|^2 + |\nabla z |^2  + s^2|z|^2\right)dxdt\\
	\leq M s \int_{-T+\eta}^{T-\eta }E_s(t)dt
	+ M \int_{-T}^{T}\int_\Omega e^{2s\varphi} | \square z |^2\,dxdt.
\end{multline}

\noindent \textbf{Step 3: Conclusion}

Using the power of $s$ in the left hand side of (\ref{eq3}) and estimates \eqref{firstterm1} and \eqref{firstterm}, taking $s$ large enough, we obtain
\begin{multline*}
		s\int_{-T+\eta}^{T-\eta}\int_{\Omega} e^{2s\varphi}\left(|\partial_t z|^2 + |\nabla z |^2  + s^2|z|^2\right)dxdt		
		\\
		\leq M\int_{-T}^T \int_\Omega e^{2s\varphi} |\square z|^2 \,dxdt 
		+Ms\int_{-T}^T \int_{\Gamma_0} e^{2s\varphi} |\partial_\nu z|^2\,d\sigma dt.
\end{multline*}
Using then estimates \eqref{EstimeeEnergieT} and \eqref{EstimeeEnergie-T}, we immediately deduce \eqref{Carlemgral}.
\end{proof}

\begin{remark}\label{RemarkChi}
	Note that, following carefully the above proof, the Carleman estimate \eqref{Carlemgral} can actually be slightly improved into
	\begin{equation}
	\begin{aligned}\label{eq2-bis}	
		s\int_{-T}^T \int_\Omega &e^{2s\varphi}\left( |\partial_t z|^2 + |\nabla z|^2 \right)\,dxdt
		+s^3\int_{-T}^T \int_\Omega e^{2s\varphi} |z|^2\,dxdt\\ 
		&\leq M \int_{-T}^T \int_\Omega e^{2s\varphi} |\square  z|^2 \,dxdt
		+M s\int_{-T}^T \int_{\Gamma_0} e^{2s\varphi} \chi(t)^2 |\partial_\nu z|^2\,d\sigma dt, 
	\end{aligned}
	\end{equation}
	where $\chi$ is the cut-off function defined in \eqref{chi}.
\end{remark}

\subsection{A Carleman estimate with pointwise term in time $-T$}\label{Sec2.4}
The proof of Theorem \ref{Carlemangral} easily gives furthermore an additional weighted estimate of the solution at time $-T$:

\begin{corollary}\label{CarlemanT}
Under the conditions of Theorem \ref{Carlemangral}, we also have
\begin{multline}\label{CarlemT}	
		s  \int_{\Omega} e^{2s\varphi(-T)}\left(|\partial_t z(-T)|^2 + |\nabla z (-T)|^2\right)dx
		+ s^3 \int_{\Omega} e^{2s\varphi(-T)}|z(-T)|^2\,dx \\
		\leq M\int_{-T}^{T} \int_{\Omega} e^{2s\varphi}| \square z|^2\,dxdt 
		+ Ms \int_{-T}^{T} \int_{\Gamma_{0}} e^{2s\varphi} \left|\partial_\nu z\right|^2\,d\sigma dt ,
\end{multline}
for all $z\in L^2((-T,T);H_0^1(\Omega))$ satisfying $\square z \in L^2(\Omega\times (-T,T))$ and $\partial_\nu z \in L^2(\partial\Omega\times (-T,T))$.
\end{corollary}

\begin{proof}
Using inequalities \eqref{rose} at $t=-T$ and \eqref{orange}, we obtain an estimate on $E_s(-T)$. We then deduce \eqref{CarlemT} from the weighted Poincar\'e inequality of Lemma \ref{Poincare} and Theorem~\ref{Carlemangral}.
\end{proof}

Let us now conclude with the proof of Theorem~\ref{ThmCarleman} that will be our main tool for the study in Section~\ref{SecCont} on the design of a constructive process for building controls that depend weakly on the potentials.
\begin{proof}[Proof of Theorem~\ref{ThmCarleman}]
The Carleman estimate of Theorem~\ref{ThmCarleman} for the operator $\square + p$ with $p\in L^\infty_{\leq m} (\Omega \times (-T,T))$ is a direct consequence of Theorem \ref{Carlemangral} and Corollary \ref{CarlemanT} noticing that in $\Omega\times(-T,T)$,
$$
	|\square z|^2 \leq 2 |\square z+ pz|^2 + 2 \|p\|_{L^\infty (\Omega \times (-T,T))}^2|z|^2 
	\leq 2 |\square z+ pz|^2 + 2 m^2 |z|^2.
$$
Then choosing $s_0$ large enough, one can absorb the term
$$
	2 M m^2 \int_{-T}^{T} \int_{\Omega} e^{2s\varphi}|z|^2\,dxdt
$$
by the left hand side of the sum of \eqref{Carlemgral} and \eqref{CarlemT}, thus obtaining \eqref{CarlemTp} with slightly different constants.
\end{proof}

\begin{remark}\label{RemarkChi-2}
	Using Remark \ref{RemarkChi}, one can easily adapt the above proof to estimate the left hand-side of \eqref{CarlemTp} by
	\begin{equation}
	\label{RHS-bis}	
		M \int_{-T}^T \int_\Omega e^{2s\varphi} |\square  z + pz|^2 \,dxdt
		+M s\int_{-T}^T \int_{\Gamma_0} e^{2s\varphi} \chi(t)^2 |\partial_\nu z|^2\,d\sigma dt, 
	\end{equation}
	where $\chi$ is the cut-off function defined in \eqref{chi}.
\end{remark}

\subsection{A Carleman estimate with pointwise term in time $0$}
We are now interested in deriving a Carleman-type estimate with $\partial_t z(\cdot,0)$ in the left-hand side under the condition that $z(x,0) = 0$ for all $x\in\Omega$.
We aim at proving Theorem~\ref{ThmCarleman-t=0} at the end.

\begin{theorem}\label{Carleman0}
Assume the multiplier condition \eqref{GCC-multiplier} and the time condition  \eqref{GCC-Time}. Define the weight functions $\varphi$ as in \eqref{poids} with $\beta \in (0,1)$ being such that \eqref{GCC-Time-Beta} holds.

Then there exist $s_{0}>0$, $\lambda>0$ and a positive constant $M$ 
such that for all $s\geq s_{0}$:
\begin{gather}
	s^{1/2} \int_{\Omega} e^{2s\varphi(0)}|\partial_t z(0)|^2 \,dx 
	+ s \int_{-T}^{T} \int_{\Omega} e^{2s\varphi}(|\partial_t z|^2 + |\nabla z |^2)\,dxdt 
	+ s^3 \int_{-T}^{T} \int_{\Omega} e^{2s\varphi}|z|^2\,dxdt \nonumber\\
	\leq M \int_{-T}^{T} \int_{\Omega} e^{2s\varphi}| \square z|^2\,dxdt 
	+ Ms \int_{-T}^{T} \int_{\Gamma_{0}} e^{2s\varphi} \left|\partial_\nu z\right|^2\,d\sigma dt\label{Carlem0}
\end{gather}
for all $z\in L^2((-T,T);H_0^1(\Omega))$ satisfying $\square z \in L^2(\Omega\times (-T,T))$, $\partial_\nu z \in L^2(\partial\Omega\times (-T,T))$ and $z(x,0) =0$ for all  $x\in\Omega$.
\end{theorem}

\begin{remark} 
 Let us emphasize the assumption $z(\cdot, 0) = 0$ in $\Omega$. Without this condition, energy estimates based on (2.24) only yield 
	$$
		E_s(0) \leq M \int_{-T}^{-T} \int_\Omega e^{2 s \varphi} |\Box z|^2 \, dx  dt + Ms \int_{-T}^T \int_{\Gamma_0} e^{2 s \varphi} |\partial_\nu z|^2 \, d \sigma d t, 
	$$
	which is not enough to our purpose, see Section 4.
\end{remark}

\begin{proof}
We consider a function $z\in L^2((-T,T);H^1_0(\Omega))$ such that $z(x,0) =0$ for all  $x\in\Omega$. We use the notations previously introduced by \eqref{P1} and \eqref{chi} and set 
$$
	w=e^{s\varphi} \chi z \quad \tn{ and } \quad P_{1}w=\partial_t^2w-\Delta w+s^2\lambda^2\varphi^2w(|\partial_t\psi|^2-|\nabla\psi|^2).
$$

Since we are interested only in the dependence of $s$, as before, all the powers of $\lambda$ and the functions $\varphi$ can be omitted and enter into the constants.

Under the condition $z(\cdot,0) = 0$ in $\Omega$, we get $w(x,0) = 0$ for all $x\in\Omega$. This allows us to do the following computations
\begin{eqnarray*}
	\int_{-T}^0\int_{\Omega} P_1w\,\partial_tw\,dxdt 
	& = & \int_{-T}^0\int_{\Omega} (\partial_t^2w-\Delta w+s^2\lambda^2\varphi^2w(|\partial_t\psi|^2-|\nabla\psi|^2))\,\partial_t w\,dxdt \\
	& = & \dfrac 12\int_{\Omega}  |\partial_tw(0)|^2  \,dx
	- \dfrac {s^2\lambda^2}2  \int_{-T}^0\int_{\Omega}  |w|^2\partial_t\left(\varphi^2(|\partial_t^2\psi|^2-|\nabla\psi|^2)\right)\,dxdt\\
	& \geq &\dfrac 12\int_{\Omega} |\partial_tw(0)|^2  \,dx  - C_\lambda s^2  \int_{-T}^0\int_{\Omega}   |w|^2 \,dxdt.
\end{eqnarray*}
implying in particular, by Cauchy-Schwarz, that
$$	
	s^{1/2} \int_{\Omega} |\partial_tw(0)|^2  \,dx 
		\leq	\int_{-T}^T\int_{\Omega} |P_1w|^2\,dxdt 
		+s \int_{-T}^T\int_{\Omega} |\partial_tw|^2\,dxdt
		+C_\lambda s^{5/2}   \int_{-T}^T\int_{\Omega}   |w|^2 \,dxdt.
$$

Moreover, $v  = we^{-s \varphi} = \chi z$ satisfies the assumption of Theorem \ref{CarlemanClas}. Therefore we can use estimate \eqref{CarlemW} on $w$ and, bounding each $\varphi^k$ from above and from below, we get:
\begin{gather}
	s^{1/2} \int_{\Omega} |\partial_tw(0)|^2  \,dx 
	+ s\int_{-T}^{T} \int_{\Omega} (| \partial_t w|^2+|\nabla w|^2)\,dxdt
	+ s^3 \int_{-T}^{T} \int_{\Omega}  |w|^2\,dxdt
	\nonumber\\
	+\int_{-T}^{T} \int_{\Omega}|P_1 w|^2\,dxdt+\int_{-T}^{T} \int_{\Omega}|P_2 w|^2\,dxdt\label{CarlemW-0}\\
	\leq M\int_{-T}^{T} \int_{\Omega}| Pw|^2\,dxdt 
	+Ms \int_{-T}^{T} \int_{\Gamma_0} \left|\partial_\nu w \right|^2  \,d\sigma dt. \nonumber
\end{gather}
Now, arguing as in Step 3 of the proof of Theorem \ref{CarlemanClas}, we obtain
\begin{gather*}
	s^{1/2} \int_{\Omega} e^{2s\varphi(0)}|\partial_t z(0)|^2 \,dx 
	+ s \int_{-T}^{T} \int_{\Omega} e^{2s\varphi}(|\partial_t (\chi z)|^2 + |\nabla (\chi z) |^2)\,dxdt 
	+ s^3 \int_{-T}^{T} \int_{\Omega} e^{2s\varphi}\chi^2 |z|^2\,dxdt \nonumber\\
	\leq M \int_{-T}^{T} \int_{\Omega} e^{2s\varphi}| \square (\chi z)|^2\,dxdt 
	+ Ms \int_{-T}^{T} \int_{\Gamma_{0}} e^{2s\varphi} \chi^2 \left|\partial_\nu z\right|^2\,d\sigma dt.
\end{gather*}
We then use energy estimates as in Theorem \ref{Carlemangral} to deduce \eqref{Carlem0}.
\end{proof}

Let us now conclude with the proof of Theorem~\ref{ThmCarleman-t=0}, main tool to study the inverse problem in Section~\ref{SecInversePb}.
\begin{proof}[Proof of Theorem~\ref{ThmCarleman-t=0}]
Since $ p\in L^\infty_{\leq m}(\Omega \times (-T,T))$, thanks to Theorem~\ref{Carleman0} and arguing as in the proof of Theorem \ref{ThmCarleman}, the potential in \eqref{Carlem0p} can be absorbed by taking $s$ large enough.

When the potential $q$ in the operator does not depend on time, one can extend the function $z$ by $z(\cdot, t) = z(\cdot, -t)$ for $t \in (-T, 0)$ and apply \eqref{Carlem0p} to this extended function $z$. Of course, since each term is odd or even, the integrals on $(-T,T)$ simply are twice the integrals on $(0,T)$, which concludes the proof of \eqref{CarlemT-0}.
\end{proof}

\begin{remark}\label{RemarkChi-3}
	Using Remark \ref{RemarkChi}, similarly as in Remark \ref{RemarkChi-2}, one can estimate the left hand side of \eqref{Carlem0p} by \eqref{RHS-bis}.
\end{remark}

\section{Application to a controllability problem}\label{SecCont}
In this section, our goal is to present what are the consequences of the Carleman estimate of Theorem \ref{ThmCarleman} with respect to the control properties of equation \eqref{EqYControl}.

In all this section, we shall assume that conditions \eqref{GCC-multiplier}--\eqref{GCC-Time} on $\Gamma_0, \,T$ hold. Then there exists $\beta \in (0,1)$ such that \eqref{GCC-Time-Beta} holds. We fix $\beta$ this way and take $\lambda$ large enough so that Theorem~\ref{ThmCarleman} applies.

\subsection{Setting}
Let us recall that the exact controllability problem under consideration is the one described in the introduction by \eqref{EqYControl}-\eqref{ExactControllability}. 
In order to solve that problem, following the duality technique introduced in \cite{FursikovImanuvilov} for parabolic equations (and that can be seen as an extension of the usual Hilbert Uniqueness Method \cite{Lions}), the idea is to minimize the functional $K_{s,p}$ defined by \eqref{K-functional}.\\

Before going further, let us take some time to describe the space on which $K_{s,p}$ is defined. In the introduction, for $p \in L^\infty(\Omega \times (-T,T))$,  we defined $K_{s,p}$ on 
\begin{multline}
	\label{Space-T-p}
	\mathcal{T}[p] = \Big\{ z \in L^2(-T,T; H^1_0(\Omega)), \hbox{ with }  (\partial_t ^2 - \Delta + p) z \in L^2(\Omega \times (-T,T))
	\\
	 \hbox{ and } \partial_\nu z \in L^2(\Gamma_0 \times (-T,T))\Big\}. 
\end{multline}
Note that this is a space of trajectories of the wave operator with potential $ (\partial_t ^2 - \Delta + p) $ and therefore it {\it a priori} depends on $p$.
In order to study the functional $K_{s,p}$, natural semi-norms on $\mathcal{T}[p]$ are the following ones:
\begin{equation}
	\label{Norm-T-s-p}
	\| z\|^2_{\textnormal{obs},s,p} = \frac{1}{s} \int_{-T}^{T}\int_\Omega e^{2s\varphi} |\partial_{t}^2 z- \Delta z + pz|^2 \,dxdt
	+\int_{-T}^{T}\int_{\Gamma_0} e^{2s\varphi}|\partial_\nu z|^2 \,d\sigma dt.
\end{equation}
Let us explain below that these quantities define norms for any parameter $s >0$ and potential $p \in L^\infty (\Omega \times (-T,T))$ and give some of their basic properties:

\begin{itemize}
\item  For all $s >0$ and $p \in L^\infty(\Omega \times (-T,T))$, and for all $z \in \mathcal T[p]$, the quantity $\| z \|_{\textnormal{obs},s,p}^2$ is equivalent to the quantity
	\begin{equation}
		\label{Norm-T-withoutweight}
		\| z\|^2_{\textnormal{obs},p} 
		=  \int_{-T}^{T}\int_\Omega |\partial_{t}^2 z - \Delta z + pz|^2 \,dxdt
		+ \int_{-T}^{T}\int_{\Gamma_0} |\partial_\nu z|^2\,d\sigma dt, 
	\end{equation}
	in the sense that there exists $C_s$ such that for all $z \in \mathcal T[p]$,
	\begin{equation}
		\label{EquivalenceNorms-ObsUsuel}
		\frac{1}{C_s} 	\| z\|_{\textnormal{obs},p}^2 \leq 	\| z\|^2_{\textnormal{obs},s,p} \leq C_s 	\| z\|^2_{\textnormal{obs},p}.
	\end{equation}
	This is a consequence of the fact that the weight function $e^{2s\varphi}$ is bounded from below and from above by positive constants depending on $s$.
	
\item  For $p \in L^\infty(\Omega \times (-T,T))$ with $s \geq s_0 (\| p\|_{L^\infty(\Omega \times (-T,T))})$ given by Theorem \ref{ThmCarleman},
 the quantity defined in \eqref{Norm-T-s-p} is indeed a norm: the Carleman estimate \eqref{CarlemTp} shows that if 
 $\| z\|_{\textnormal{obs},s,p}  = 0$, 
 then $z \equiv 0$ and that $	\| z\|_{\textnormal{obs},s,p} $ measures the $L^2(-T,T;H^1_0(\Omega))$-norm of $z$.
This proves that $\| \cdot\|_{\textnormal{obs},s,p} $  is a norm on $\mathcal{T}[p]$ 
for all $s \geq s_0 (\| p\|_{L^\infty(\Omega \times (-T,T))})$. According to the first item, this is actually true for all $s >0$.
	
\item $\mathcal{T}[p] = \mathcal{T}[0]$ for all $p \in L^\infty(\Omega \times (-T,T))$. This is due to the fact that $p z \in L^2(\Omega \times (-T,T))$  when $p \in L^\infty(\Omega \times (-T,T))$ and $z \in L^2(-T,T; H^1_0(\Omega))$.

For convenience, we shall now denote $\mathcal{T}[0]$ simply by $\mathcal{T}$, where
\begin{multline}
	\label{Space-T}
	\mathcal{T} = \Big\{ z \in L^2(-T,T; H^1_0(\Omega)), \hbox{ with }  (\partial_t ^2 - \Delta ) z \in L^2(\Omega \times (-T,T))
	\\
	 \hbox{ and } \partial_\nu z \in L^2(\Gamma_0 \times (-T,T))\Big\}. 
\end{multline}
	
\item  For all $m>0$, there exist constants $M(m)>0$ and $s_0(m) >0$ independent of $s$ and $p$ such that for all $s\geq s_0(m)$ and $p^a, p^b \in L^\infty_{\leq m} (\Omega \times (-T,T))$, 
	\begin{equation}
		\label{EquivalentNorms-T-s-p}
		\frac{1}{M} \| z\|_{\textnormal{obs},s,p^b}\leq	\| z\|_{\textnormal{obs},s,p^a}  \leq M	\| z\|_{\textnormal{obs},s,p^b}.
	\end{equation}
	This result follows immediately from Theorem \ref{ThmCarleman} and its proof.
	Note that in \eqref{EquivalentNorms-T-s-p}, these equivalences of norms are proven uniformly with respect to $s\geq s_0(m)$ for potentials lying in $ L^\infty_{\leq m} (\Omega \times (-T,T))$. This is an important remark.\\
\end{itemize}

In the following, we will denote by $\left( \mathcal{T}, \| \cdot \|_{\textnormal{obs},s,p} \right)$ the space $\mathcal{T}$ endowed with the norm $\| \cdot \|_{\textnormal{obs},s,p}$. 
But we drive the attention of the reader to the fact that these norms $\| \cdot \|_{\textnormal{obs},s,p}$ are not uniformly equivalent with respect to $s>0$.\\

We now introduce the functional space $ H^1_0(\Omega) \times L^2(\Omega)  $
on which we consider the family of norms, 
$$
	\|(z_0, z_1) \|_{-T, s}^2 =  \int_{\Omega} e^{2s\varphi(-T)}\left( |\nabla z_0|^2+|z_1|^2\right) dx,\quad \tn{ for }s >0.
$$
Again, several remarks can be done:

\begin{itemize}
\item Using the weighted Poincar\'e inequality of Lemma \ref{Poincare}, $\|\cdot \|_{-T, s}$ is equivalent, uniformly in $s>0$, to 
	$$
	 \left(\int_{\Omega} e^{2s\varphi(-T)}\left( |\nabla z_0|^2+|z_1|^2\right)dx
		+ s^2 \int_{\Omega} e^{2s\varphi(-T)}|z_0|^2\,dx\right)^{1/2}.
	$$

\item According to Theorem \ref{ThmCarleman}, for all $z \in \mathcal{T}$, $(z(-T), \partial_t z(-T))$ belongs to $H^1_0(\Omega) \times L^2(\Omega) $ and for all $m>0$, there exists a constant $C$ independent of $p$ and $s$ such that for all $p \in L^\infty_{\leq m} (\Omega \times (-T,T))$ and $s \geq s_0(m)$, 
\begin{equation}
	\label{EstCarlemanForCoercivity}
	\|(z(-T), \partial_t z(-T)) \|_{-T, s} \leq C \| z\|_{\textnormal{obs},s,p}, \quad z \in \mathcal{T}.
\end{equation}	
\end{itemize}

Note that, anyway, for all $s >0$ and $p \in L^\infty(\Omega \times (-T,T))$, one easily checks that, bounding the functions depending on $s$ if needed, there exists $C(s,p)$ such that 
\begin{equation}
	\label{EstCarlemanForCoercivity-weak}
		\|(z(-T), \partial_t z(-T)) \|_{-T, s} \leq C(s,p) \| z\|_{\textnormal{obs},s,p}, \quad z \in \mathcal{T}.
\end{equation}

Finally, we also introduce the dual space $L^2(\Omega) \times H^{-1}(\Omega)$ on which we consider the family of norms
\begin{equation}
	\label{NormeDuale}
	\|(y_0, y_1) \|_{-T, s, *}^2 =  \int_{\Omega} e^{-2 s \varphi(-T)}\left(|y_0|^2 + \left|\nabla (-\Delta_d)^{-1} y_1\right|^2\right)dx.
\end{equation}

\subsection{Construction of a null-controlled trajectory}

\begin{proposition}
	\label{PropControle}
	Assume the multiplier condition \eqref{GCC-multiplier} and the time condition \eqref{GCC-Time-Beta}.
	
	Then, for all $s >0$ and $p \in L^\infty(\Omega \times (-T,T))$, the functional $K_{s,p}$ defined in \eqref{K-functional}, that we recall below for the convenience of the reader,
	\begin{multline*}
	K_{s,p} (z)= \frac{1}{2s}\int^T_{-T}\int_\Omega e^{2s\varphi} | \partial_t^2 z-\Delta z+p z|^2\,dxdt +\frac{1}{2}\int^T_{-T}\int_{\Gamma_0}e^{2s\varphi}  |\partial_\nu z|^2 \,d\sigma dt \\
			+ \langle (y_0^{-T},y_1^{-T}),(z(-T),\partial_t z(-T))\rangle_{(L^2\times H^{-1})\times(H^1_0\times L^2)}, 
	\end{multline*}
	 is continuous, strictly convex and coercive on $\left( \mathcal{T}, \| \cdot \|_{\textnormal{obs},s,p}\right)$ for initial data
	$(y_0^{-T},y_1^{-T})$ in $\left( L^2(\Omega) \times H^{-1}(\Omega),\|(\cdot, \cdot) \|_{-T, s, *}\right)$ 
	and therefore admits a unique minimizer $Z[s,p] \in \mathcal{T}$. 
	
	Setting 
$$
	Y[s,p] = \frac{1}{s}e^{2 s \varphi} (\partial_t^2 - \Delta +p) Z[s,p] \quad \text{ and } \quad U[s,p] = e^{2 s \varphi} \partial_\nu Z[s,p] \fcar_{\Gamma_0}
$$ 
as in \eqref{LinkControlPb-Adjoint}, $Y[s,p]$ solves \eqref{EqYControl} with control $U[s,p]$ and satisfies the control requirement \eqref{ExactControllability}.
	 
	 Besides, for all $m >0$, there exists a constant $M>0$ independent of $s$ and $p$ such that for all $p \in L^\infty_{\leq m}(\Omega \times (-T,T))$, $s \geq s_0(m)$, for all data $(y_0^{-T}, y_1^{-T}) \in L^2(\Omega) \times H^{-1}(\Omega)$, the minimizer $Z[s,p]$ of $K_{s,p}$ satisfies:
	 \begin{multline}
	\| Z[s,p] \|_{\textnormal{obs},s,p}^2 =  
	s\ds \int_{-T}^{T} \int_\Omega e^{-2s\varphi}|Y[s,p]|^2 \,dxdt
	+  
	\int_{-T}^{T} \int_{\Gamma_0} e^{-2s\varphi} |U[s,p] |^2\,dxdt 
	\\
	\leq M \|(y_0^{-T}, y_1^{-T}) \|_{-T, s, *}^2.
	\label{EstControlTraj}
	\end{multline}
\end{proposition}

\begin{proof}
We fix $s >0$ and $p \in L^\infty(\Omega \times (-T,T))$.
For $(y_0^{-T}, y_1^{-T}) \in  L^2(\Omega) \times H^{-1}(\Omega)$, the functional $K_{s,p}$ is defined and continuous on 
$\left( \mathcal{T}, \| \cdot \|_{\textnormal{obs},s,p} \right)$ because of \eqref{EstCarlemanForCoercivity-weak}. Estimate \eqref{EstCarlemanForCoercivity-weak} also yields immediately the coercivity of $K_{s,p}$:
\begin{align*}
		K_{s,p}(z) &\geq \dfrac 12 \|z\|^2_{\textnormal{obs},s,p} - \|(z_0^{-T},z_1^{-T})\|_{-T,s} \|(y_0^{-T}, y_1^{-T}) \|_{-T, s, *}
		\\
		&
		\geq \dfrac 12 \|z\|^2_{\textnormal{obs},s,p} - C(s,p)  \|z\|_{\textnormal{obs},s,p}\|(y_0^{-T}, y_1^{-T}) \|_{-T, s, *}.
\end{align*}
Therefore $K_{s,p}$ has a unique minimizer $Z[s,p]$ on $ \left( \mathcal{T},\| \cdot \|_{\textnormal{obs},s,p}\right)$ .

Since $K_{s,p}(0)=0$, we have $K_{s,p}(Z)\leq 0$, which, according to the above inequality, implies
\begin{equation*}
	\|Z[s,p]\|_{\textnormal{obs},s,p} \leq 2 C(s,p) \|(y_0^{-T}, y_1^{-T}) \|_{-T, s, *}.
\end{equation*}

Besides, when for some $m>0$, we have $s \geq s_0(m)$ and $p \in L^\infty_{\leq m} (\Omega \times (-T,T))$, using \eqref{EstCarlemanForCoercivity} instead of \eqref{EstCarlemanForCoercivity-weak}, the above constant $C$ can be chosen independently of $s \geq s_0(m)$ and $p \in L^\infty_{\leq m} (\Omega \times (-T,T))$:
\begin{equation*}
	\|Z[s,p]\|_{\textnormal{obs},s,p} \leq 2 C \|(y_0^{-T}, y_1^{-T}) \|_{-T, s, *}, 
\end{equation*}
which is precisely \eqref{EstControlTraj}.\\

Let us now check that $Y[s,p]$ defined by \eqref{LinkControlPb-Adjoint} is a controlled trajectory of \eqref{EqYControl} with control function $U[s,p]$ that satisfies the control requirement \eqref{ExactControllability}.
In order to simplify the notations, until the end of the proof, we fix $s>0$ and $p \in L^\infty(\Omega \times (-T,T))$ and denote $Z[s,p], Y[s,p], U[s,p]$ by $Z,Y,U$.

The Euler-Lagrange equation given by the minimization of $K_{s,p}$ is as follows: for all $z\in \mathcal{T}$,
\begin{multline}\label{EL}
		\frac 1s\int_{-T}^{T} \int_{\Omega} e^{2s\varphi} (\partial_{t}^2 z - \Delta z+pz) (\partial_t^2 Z- \Delta Z +pZ)\,dxdt
		+\int_{-T}^{T} \int_{\Gamma_0} e^{2s\varphi}\partial_\nu z  \partial_\nu Z \,d\sigma dt\\
		+ \langle (y_0^{-T},y_1^{-T}),(z(-T),\partial_t z(-T))\rangle_{(L^2\times H^{-1})\times(H^1_0\times L^2)} = 0.
\end{multline}
Therefore, with $Y$ and $U$ as in \eqref{LinkControlPb-Adjoint}, we obtain, for all $z \in \mathcal{T}$,
\begin{multline*}
	\int_{-T}^{T} \int_{\Omega}  (\partial_{t}^2 z - \Delta z+pz)Y \,dxdt
	+\int_{-T}^{T} \int_{\Gamma_0} \partial_\nu z  U \,d\sigma dt
	\\
	+ \langle (y_0^{-T},y_1^{-T}),(z(-T),\partial_t z(-T))\rangle_{(L^2\times H^{-1})\times(H^1_0\times L^2)} = 0.
\end{multline*}
But this  is precisely the dual formulation of equation \eqref{EqYControl} and integrations by parts yield, for all $z \in \mathcal{T}$,
\begin{multline*}
	 \int_{-T}^{T} \int_{\Omega}  z(\partial_{t}^2 - \Delta + p)Y \,dxdt
	+\int_{-T}^{T} \int_{\partial\Omega} \partial_\nu z (U \fcar_{\Gamma_0}- Y ) \,d\sigma dt
	\\
	 + \left\langle \left( (y_0^{-T},y_1^{-T}) - (Y(-T), \partial_t Y(-T)\right),(z(-T), \partial_t z(-T))\right\rangle_{(L^2\times H^{-1})\times(H^1_0\times L^2)} 
	\\
	 + \left\langle (Y(T), \partial_t Y(T),(z(T),\partial_t z(T))\right\rangle_{(L^2\times H^{-1})\times(H^1_0\times L^2)} 
	 =0 .
\end{multline*}
This implies that $Y = Y[s,p]$ solves \eqref{EqYControl}--\eqref{ExactControllability} with control function $U = U[s,p]$.
\end{proof}

\subsection{Dependence of the controls with respect to the potentials}\label{SecProofThmControl}
We are now in position to prove Theorem \ref{ThmErrorRelative}.

\begin{proof}[Proof of Theorem \ref{ThmErrorRelative}]
In order to simplify the notations, we set $L^i=\partial_t^2 - \Delta + p^i$, $Z^i=Z[s,p^i]$, $Y^i = Y[s,p^i]$ and $U^i=U[s,p^i]$ for $i \in \{a, b\}$.
Using \eqref{LinkControlPb-Adjoint}, we are going to bound the following expression:
	\begin{align}
		s \int_{-T}^{T} \int_{\Omega} &e^{-2s\varphi}|Y^a - Y^b|^2\,dxdt 
		+  \int_{-T}^{T} \int_{\Gamma_0} e^{-2s\varphi} |U^a - U^b |^2\,dxdt \nonumber\\
		& = \frac 1s \int_{-T}^{T} \int_{\Omega} e^{2s\varphi} |L^a Z^a - L^b Z^b|^2 \,dxdt
		+  \int_{-T}^{T} \int_{\Gamma_0} e^{2s\varphi} |\partial_\nu Z^a - \partial_\nu Z^b|^2\,dxdt \nonumber\\
		& = \frac 1s\int_{-T}^{T} \int_{\Omega} e^{2s\varphi} \left(|L^a Z^a|^2 + |L^b Z^b|^2 - 2 L^aZ^a L^bZ^b\right)dxdt\nonumber\\
		&\quad +  \int_{-T}^{T} \int_{\Gamma_0} e^{2s\varphi} \left(|\partial_\nu Z^a|^2 + |\partial_\nu Z^b|^2 - 2\partial_\nu Z^a\partial_\nu Z^b\right)dxdt.
		\label{star}
	\end{align}
We recall first the Euler-Lagrange equation \eqref{EL} associated to the trajectory $Z^i$, for $i \in \{a,b\}$:
\begin{multline}
\label{eq4}
	 	\frac 1s\int_{-T}^{T} \int_{\Omega} e^{2s\varphi} L^iz L^iZ^i\,dxdt
		+\int_{-T}^{T} \int_{\Gamma_0} e^{2s\varphi}\partial_\nu z  \partial_\nu Z^i \,dxdt
		\\
		+ \langle (y_0^{-T},y_1^{-T}),(z(-T),\partial_t z(-T))\rangle_{(L^2\times H^{-1})\times(H^1_0\times L^2)} =0
\end{multline}
We apply (\ref{eq4}) to $z=Z^i$ and $z=Z^j$, $j \neq i$ and obtain the following equations for $(i,j) \in \{a,b\}^2$:
	\begin{multline}\label{eq5}
	 	 \frac 1s \int_{-T}^{T} \int_{\Omega} e^{2s\varphi} |L^iZ^i|^2 \,dxdt
		+
		\int_{-T}^{T} \int_{\Gamma_0} e^{2s\varphi} |\partial_\nu Z^i|^2 \,dxdt
		\\
		+ \langle (y_0^{-T},y_1^{-T}),(Z^i(-T),\partial_t Z^i(-T))\rangle_{(L^2\times H^{-1})\times(H^1_0\times L^2)}=0;
	\end{multline}
	\begin{multline}\label{eq6}
	 	 \frac 1s \int_{-T}^{T} \int_{\Omega} e^{2s\varphi} L^i Z^j L^iZ^i\,dxdt
		+
		\int_{-T}^{T} \int_{\Gamma_0 }e^{2s\varphi} \partial_\nu Z^j  \partial_\nu Z^i \,dxdt 
		\\
		+ 
		\langle (y_0^{-T},y_1^{-T}),(Z^j(-T),\partial_t Z^j(-T))\rangle_{(L^2\times H^{-1})\times(H^1_0\times L^2)}=0.
	\end{multline}

Summing (\ref{eq5}) for $i=a$ and $i=b$ and subtracting from it (\ref{eq6}) for $(i,j)=(a,b)$ and $(i,j)=(b,a)$, we get:
	\begin{multline*}
		 \frac 1s \int_{-T}^{T} \int_{\Omega} e^{2s\varphi} \left(|L^aZ^a|^2 + |L^bZ^b|^2 
		- L^aZ^b L^aZ^a - L^bZ^a L^bZ^b\right)dxdt\\
		+\int_{-T}^{T} \int_{\Gamma_0} e^{2s\varphi} \left(|\partial_\nu Z^a|^2 
		+ |\partial_\nu Z^b|^2 -2 \partial_\nu Z^a  \partial_\nu Z^b\right)dxdt=0.
	\end{multline*}
Thus, making use of $ L^i Z^j = L^jZ^j + (p^i-p^j)Z^j $, we have from \eqref{star}:
	\begin{eqnarray*}
		&&s\int_{-T}^{T} \int_{\Omega} e^{-2s\varphi}|Y^a - Y^b|^2 \,dxdt
		+\int_{-T}^{T} \int_{\Gamma_0} e^{-2s\varphi} |U^a - U^b|^2  \,dxdt	\\
		&=&\frac 1s \int_{-T}^{T} \int_{\Omega} e^{2s\varphi} (L^aZ^b L^aZ^a 
		+ L^bZ^a L^bZ^b - 2 L^aZ^a L^bZ^b)\,dxdt\\
		&=& \frac 1s \int_{-T}^{T} \int_{\Omega} e^{2s\varphi} (p^b-p^a) (Z^b L^aZ^a -Z^a L^bZ^b)\,dxdt.	
	\end{eqnarray*}
Finally, we obtain
\begin{multline}
	\label{AlmostDone}
	s \int_{-T}^{T} \int_{\Omega}e^{-2s\varphi}|Y^a - Y^b|^2\,dxdt 
		+ \int_{-T}^{T} \int_{\Gamma_0} e^{-2s\varphi} |U^a - U^b|^2 \,dxdt	
		\leq \frac{1}{s^{5/2}} \|p^a-p^b\|_{L^\infty(\Omega \times (-T,T))}
	\\	
	 \left( 
	 \int_{-T}^{T} \int_{\Omega}e^{2s\varphi}  (|L^aZ^a|^2 + |L^bZ^b|^2) \,dxdt
	+
	s^3 \int_{-T}^{T} \int_{\Gamma_0} e^{2s\varphi}  (|Z^a|^2 + |Z^b|^2)\,dxdt \right).
\end{multline}

Applying the Carleman estimate of Theorem \ref{ThmCarleman} for $i \in \{a,b\}$, we have:
$$	
	s^3 \int_{-T}^{T}\int_\Omega e^{2s\varphi} |Z^i|^2 \,dxdt 
	\leq M \int_{-T}^T \int_\Omega e^{2s\varphi} |L^i Z^i|^2 \,dxdt 
	+M s\int_{-T}^T \int_{\Gamma_0} e^{2s\varphi} |\partial_\nu Z^i|^2\,d\sigma dt.
$$
Of course, this implies
\begin{multline*}
	 \int_{-T}^{T} \int_{\Omega}e^{2s\varphi}  |L^iZ^i|^2  \,dxdt
	+
	s^3 \int_{-T}^{T} \int_{\Omega} e^{2s\varphi}  |Z^i|^2 \,dxdt 
	\\
	\leq 
	M \int_{-T}^T \int_\Omega e^{2s\varphi} |L^i Z^i|^2 \,dxdt 
	+M s\int_{-T}^T \int_{\Gamma_0} e^{2s\varphi} |\partial_\nu Z^i|^2\,d\sigma dt
\end{multline*}
and since
\begin{eqnarray*}
	&&  \int_{-T}^T \int_\Omega e^{2s\varphi} |L^i Z^i|^2 \,dxdt 
	+ s\int_{-T}^T \int_{\Gamma_0} e^{2s\varphi} |\partial_\nu Z^i|^2\,d\sigma dt\\
	&=&  s^2 \int_{-T}^T \int_\Omega e^{-2s\varphi} |Y^i|^2 \,dxdt 
	+ s\int_{-T}^T \int_{\Gamma_0} e^{-2s\varphi} |U^i|^2\,d\sigma dt,
\end{eqnarray*}
it can be rewritten as
\begin{multline}
	\label{EstVenantDeCarl}
	 \int_{-T}^{T} \int_{\Omega}e^{2s\varphi}  |L^iZ^i|^2  \,dxdt+ s^3 \int_{-T}^{T}\int_\Omega e^{2s\varphi} |Z^i|^2 \,dxdt 
	\\
	\leq 
	M s \left( s \int_{-T}^{T} \int_{\Omega}e^{-2s\varphi}|Y^i|^2\,dxdt 
		+  \int_{-T}^{T} \int_{\Gamma_0} e^{-2s\varphi} |U^i|^2 \,dxdt\right).
\end{multline}
Combining \eqref{AlmostDone} and \eqref{EstVenantDeCarl}, we obtain the desired estimate \eqref{ErrorRelative}.
\end{proof}

\section{Application to an inverse problem}\label{SecInversePb}

Let us consider the inverse problem defined in \eqref{EqW}. In this section, we shall propose an algorithm based on the Carleman estimate \eqref{CarlemT-0} and a data assimilation approach. Let us recall that the unknown is the potential $Q = Q(x)$, that we aim at recovering from the measurement of the normal derivative of the solution $W[Q]$ of \eqref{EqW} on $\Gamma_0 \times (0,T)$. We also assume that $Q \in L^\infty_{\leq m}(\Omega)$ for some given constant $m>0$. 

Let us also mention that we are working under the geometrical assumptions \eqref{GCC-multiplier}--\eqref{GCC-Time}, and the function $\varphi$ we shall consider below always satisfies \eqref{GCC-Time-Beta} such that Theorem \ref{ThmCarleman-t=0} holds when the parameter $s$ is large enough.

As said in the introduction, one can find in \cite{Baudouin01} the proof of the fact that the additional information $\partial_\nu W[Q]$ on $\Gamma_0 \times (0,T)$ allows to identify $Q$ uniquely within the class of potentials in $L^\infty_{\leq m}(\Omega)$. Our approach will go further, providing an explicit algorithm to compute $Q$.

Our goal is indeed to prove that Algorithm \ref{Algo}, presented in Section~\ref{SubsecIntroInversePb}, is convergent when $s$ is large enough, as described in Theorem \ref{Thm-Convergence}.

In the following, we shall first present the idea underlying this algorithm. We will then focus on the proof of Theorem \ref{ThmDependenceG} which is the main step within the proof of the convergence result of Theorem \ref{Thm-Convergence}.
\subsection{The general idea}\label{SectionAlgorithm}

Algorithm \ref{Algo} is based on the fact that if $W$ is the solution of equation \eqref{EqW} and $w[q^k]$ solves \eqref{Eqwk}, 
then 
\begin{equation}
	\label{z-k-exact}
	z^k = \partial_t \left( w[q^k] - W[Q]\right)
\end{equation}
solves
\begin{equation}
	\label{EqExacteWk}
	\left\{
		\begin{array}{ll}
					\partial_{t}^2 z^k - \Delta z^k + q^k z^k = g^k, & \tn{in } \Omega \times (0,T), \\
					z^k = 0, & \tn{on } \partial \Omega \times (0,T), 
				\\
					z^k(0) = 0, \quad \partial_t z^k(0) = z_1^k, \qquad& \tn{in } \Omega,
		\end{array}
	\right.	
\end{equation}
where
\begin{equation}
	\label{TrucsEnk}
	g^k =   (Q- q^k) \partial_t W[Q], \quad  \quad z_1^k =  (Q- q^k) w_0,
\end{equation}
and by definition, 
\begin{equation}
	\label{Observation-k}
		 \mu^k = \partial_\nu z^k \tn{ on }\Gamma_0 \times (0,T).
\end{equation}
Of course, both variables $g^k$ and $z_1^k$ are unknown, but the variable $g^k$ brings lower order information than $\mu^k$. This fact is actually the milestone of the proof of Theorem \ref{Thm-Convergence}. We shall then try to approximate $z_1^k$ through the additional information \eqref{Observation-k} and let the source term free, as it is in the functional $J_{s,q^k}[\mu^k,0]$ in \eqref{FunctionalJ-k}. 

Note that this idea is behind the proofs of stability by compactness uniqueness arguments as in \cite{PuelYam96,PuelYam97,Yam99,StefanovUhlmann2011} or by Carleman estimates given in \cite{Baudouin01,ImYamIP01,ImYamCom01,ImYamIP03}.

\subsection{Study of the functional $J_{s,q}[\mu,g]$}

We  first give a functional setting for the minimization of the functional $J_{s,q}[\mu,g]$ given in \eqref{DefFunctional-J-DataAss}, that we recall below for the convenience of the reader:
\begin{equation*}
	J_{s,q}[\mu,g](z) = \frac{1}{2s} \int_0^T \int_{\Omega} e^{2s \varphi} |\partial_{t}^2 z - \Delta z + q z - g|^2\,dxdt
	+ \frac{1}{2} \int_0^T \int_{\Gamma_0} e^{2s \varphi} | \partial_\nu z - \mu |^2 \ d\sigma dt,
\end{equation*}
defined on the trajectories $z$ such that $z \in L^2(0,T;H_0^1(\Omega))$, $\partial_{t}^2 z - \Delta z + q z \in L^2(\Omega \times (0,T))$, $\partial_\nu z \in L^2(\Gamma_0 \times (0,T))$ and $z(\cdot, 0) = 0$ in $\Omega$.

Of course, it is very close to the functional $K_{s,p}$ and we shall therefore introduce the space 
\begin{multline}
	\mathcal{T}^+ = \Big\{z \in L^2(0,T; H_0^1(\Omega)), \quad \hbox{with }  \partial_{t}^2 z - \Delta z \in L^2(\Omega \times (0,T)), 
	\\
	z(\cdot , 0) = 0 \hbox{ in } \Omega \quad \hbox{ and } \partial_\nu z \in L^2(\Gamma_0 \times (0,T))\Big\}, 
\end{multline}
and the family of norms
$$
	\| z\|^2_{\textnormal{obs}^+,s,q} = \frac{1}{s} \int_{0}^{T}\int_\Omega e^{2s\varphi} |\partial_{t}^2 z- \Delta z + q z|^2 \,dxdt
	+\int_{0}^{T}\int_{\Gamma_0} e^{2s\varphi}|\partial_\nu z|^2 \,d\sigma dt.
$$
Note that these are norms for all $s>0$ and $q \in L^\infty( \Omega)$ according to Theorem \ref{ThmCarleman-t=0}.

The properties of the space $\mathcal{T}^+$ and the family of norms $\| \cdot \|_{\textnormal{obs}^+,s,q}$ are of course completely similar to the ones of $\mathcal{T}$ in \eqref{Space-T}  endowed with the family of norms $\| \cdot \|_{\textnormal{obs},s,p}$ introduced in \eqref{Norm-T-s-p}. 
Therefore, we refer the reader to Section \ref{SecCont} for remarks and comments on $\left(\mathcal{T}^+, \| \cdot\|_{\textnormal{obs}^+,s,q}\right)$.\\

The first result states the well-posedness of the minimization problem of $J_{s,q}[\mu,g]$. 
\begin{proposition}
	\label{ThmJ-DataAssimilation}
	Assume the multiplier condition \eqref{GCC-multiplier} and the time condition \eqref{GCC-Time-Beta}.
	Assume also $\mu \in L^2(\Gamma_0 \times (0,T))$ and $g \in L^2(\Omega \times (0,T))$.
	
	Then, for all $s >0$ and $q \in L^\infty(\Omega)$, the functional $J_{s,q}[\mu,g]$ defined in \eqref{DefFunctional-J-DataAss} 
 is continuous, strictly convex and coercive on 
$(\mathcal{T}^+, \|\cdot \|_{\textnormal{obs}^+,s,q})$. 
The functional $J_{s,q}[\mu,g]$ therefore admits a unique minimizer $Z$ in $\mathcal{T}^+$. 

Besides, for all data $(\mu,g) \in L^2(\partial\Omega\times(0,T)) \times L^2(\Omega \times (0,T))$, 
the minimizer $Z$ of $J_{s,q}[\mu,g]$ satisfies:
	 \begin{equation*}
	\| Z\|^2_{\textnormal{obs}^+,s,q} 
	\leq \frac{4}{s}\int_{0}^{T} \int_{\Omega} e^{2s\varphi}|g|^2\,dxdt 
		 + 4 \int_{0}^{T} \int_{\Gamma_0} e^{2s\varphi} \left|\mu\right|^2\,d\sigma dt .
	\end{equation*}
\end{proposition}

\begin{proof}
The continuity, strict convexity and  coercivity of the functional $J_{s,p}[\mu,g]$ is straightforward and left to the reader.

To get estimates on the minimizer $Z$, we use $J_{s,q}[\mu,g](Z) \leq J_{s,q}[\mu,g](0)$:
\begin{multline*}
	\frac{1}{s} \int_0^T \int_{\Omega} e^{2s \varphi} |\partial_{t}^2 Z - \Delta Z + q Z - g|^2\,dxdt
	+ 
	 \int_0^T \int_{\Gamma_0} e^{2s \varphi} | \partial_\nu Z - \mu |^2 \ d\sigma dt
	 \\
	 \leq
	 \frac{1}{s}\int_{0}^{T} \int_{\Omega} e^{2s\varphi}|g|^2\,dxdt 
		 +  \int_{0}^{T} \int_{\Gamma_0} e^{2s\varphi} \left|\mu\right|^2\,d\sigma dt .
\end{multline*}
Developing the square on the left hand side, we obtain
\begin{multline*}
	\frac{1}{s} \int_0^T \int_{\Omega} e^{2s \varphi} |\partial_{t}^2 Z - \Delta Z + q Z|^2\,dxdt
	+ 
	 \int_0^T \int_{\Gamma_0} e^{2s \varphi} | \partial_\nu Z |^2 \ d\sigma dt
	 \\
	 \leq
	\frac{2}{s} \int_0^T \int_{\Omega} e^{2s \varphi} (\partial_{t}^2 Z - \Delta Z + q Z) g\,dxdt 
	+ 
	2 \int_0^T \int_{\Gamma_0} e^{2s \varphi} \partial_\nu Z \mu \ d\sigma dt. 
\end{multline*}
Using $2 ab \leq a^2/2 + 2 b^2$, we thus obtain
	 \begin{equation*}
	\frac{1}{2}\| Z\|^2_{\textnormal{obs}^+,s,q} 
	\leq \frac{2}{s}\int_{0}^{T} \int_{\Omega} e^{2s\varphi}|g|^2\,dxdt 
		 + 2 \int_{0}^{T} \int_{\Gamma_0} e^{2s\varphi} \left|\mu\right|^2\,d\sigma dt,
	\end{equation*}
which concludes the proof of Proposition \ref{ThmJ-DataAssimilation}.
\end{proof}

Of course, our goal is not only to prove that the functional $J_{s,q}[\mu,g]$ has a minimum, but rather to study how the minimum of $J_{s,q}[\mu,g]$ depends on the source term $g$. Indeed, $z^k$ in \eqref{z-k-exact} is the minimum of the functional $J_{s,q^k}[\mu^k,g^k]$, whatever $s>0$ is, whereas in the algorithm, $Z^k$ is  the minimizer of the functional $J_{s,q^k}[\mu^k,0]$, see \eqref{FunctionalJ-k}.

This is precisely the goal of Theorem \ref{ThmDependenceG}. As in Section \ref{SecProofThmControl},  we shall rely on the Euler Lagrange equations satisfied by the minimum of the functionals $J_{s,q}[\mu, g^a]$ and $J_{s,q}[\mu,g^b]$. 

\begin{proof}[Proof of Theorem \ref{ThmDependenceG}] 
Let us write the Euler Lagrange equations satisfied by $Z^j$, for $j \in \{a, b\}$: 
\begin{multline}
	\label{EulerLagrangeIdentity}
	\frac{1}{s} \int_0^T \int_{\Omega} e^{2s \varphi} (\partial_t^2 Z^j - \Delta Z^j + q Z^j - g^j) (\partial_t^2 z - \Delta z + q z) \,dxdt\\
	+ \int_0^T \int_{\Gamma_0} e^{2s \varphi} (\partial_\nu Z^j - \mu) \partial_\nu z \, d\sigma dt =0,
\end{multline}
for all $z\in \mathcal{T}^+$. Applying \eqref{EulerLagrangeIdentity} for $j= a$ and $j= b$ to $z = Z^a - Z^b$  and subtracting the two identities, we obtain:
\begin{multline*}
	 \frac{1}{s} \int_0^T \int_{\Omega} e^{2 s \varphi} |\partial_t^2 z - \Delta z + q z|^2\, dxdt  
	+ \int_0^T \int_{\Gamma_0} e^{2s \varphi} |\partial_\nu z|^2 \, d\sigma dt\\
	= \frac{1}{s} \int_0^T \int_{\Omega} e^{2 s \varphi} (g^a - g^b) (\partial_t^2 z - \Delta z + q z) \,dxdt .
\end{multline*} 
This implies that
\begin{multline}
	\label{EstErrorsF}
	\frac{1}{2} \int_0^T \int_{\Omega} e^{2 s \varphi}  |\partial_t^2 z - \Delta z + q z|^2 \,dxdt  
	+ s \int_0^T \int_{\Gamma_0} e^{2s \varphi} |\partial_\nu z|^2 \, d\sigma dt 	\\
	\leq \frac{1}{2}  \int_0^T \int_{\Omega} e^{2 s \varphi} |g^a- g^b|^2\,dxdt .
\end{multline}
But the left hand side of \eqref{EstErrorsF} precisely is the right hand side of the Carleman estimate \eqref{CarlemT-0}. Hence, applying Theorem \ref{ThmCarleman-t=0} to $z$, we immediately deduce \eqref{EstMin-s}.
\end{proof}

\subsection{Convergence of Algorithm 1}\label{SectionConvergence}
Let us now focus on the proof of Theorem \ref{Thm-Convergence}.
\begin{proof}[Proof of Theorem \ref{Thm-Convergence}] We use Theorem \ref{ThmDependenceG} since, as we explained, we have to compare the minimum $Z^k$ of $J_{s,q^k}[\mu^k, 0]$ with $z^k = \partial_t (w[q^k] - W[Q])$ solution of  \eqref{EqExacteWk}, which corresponds to the minimum of $J_{s,q^k}[\mu^k, g^k]$. Note that this requires $q^k \in L^\infty_{\leq m} (\Omega)$, which is guaranteed at each step of the algorithm by \eqref{Q-k+1}. We obtain
\begin{equation}
	\label{Est-Diff1}
	s^{1/2} \int_\Omega e^{2 s \varphi(0)} |\partial_t Z^k(0) - \partial_t z^k(0) |^2 \, dx  \leq M \int_0^T \int_{\Omega } e^{2 s \varphi} |g^k|^2\, dxdt.
\end{equation}
But, from \eqref{Q-k-tilde} and \eqref{TrucsEnk}, 
$$
	\partial_t Z^k(\cdot,0) =( \tilde q^{k+1} - q^k) w_0, \quad \partial_t z^k(\cdot,0) = (Q - q^k) w_0 \quad \tn{and}\quad g^k = (Q - q^k) \partial_t W[Q].
$$
Therefore, since $\varphi(\cdot, t)\leq\varphi(\cdot,0)$ for all $t\in(0,T)$, estimate \eqref{Est-Diff1} reads:
$$
	s^{1/2}\int_\Omega  e^{2 s \varphi(0)} |w_0|^2 (\tilde q^{k+1} - Q)^2\, dx \leq M \|\partial_t W[Q]\|_{L^2(0,T; L^\infty(\Omega))}^2 \int_{\Omega} e^{2 s \varphi(0)} (q^k - Q)^2\, dx.
$$
Of course, using the strict positivity \eqref{Positivity} of $w_0$, this yields in particular that
\begin{equation}
	\label{EstimeesErreur-Tilde}
	\int_\Omega e^{2 s \varphi( 0)} (\tilde q^{k+1} - Q)^2\, dx \leq \frac{M}{s^{1/2}} \frac{\|\partial_t W[Q]\|_{L^2(0,T; L^\infty(\Omega))}^2}{\left( \inf_\Omega |w_0| \right)^2} \int_{\Omega} e^{2 s \varphi( 0)} (q^k - Q)^2\, dx.
\end{equation}
Since $T_m$ defined in \eqref{Q-k+1} is Lipschitz and $T_m(Q) = Q$ (because $Q \in L^\infty_{\leq m} (\Omega)$), we have 
$| q^{k+1} - Q|  = |T_m(\tilde q^{k+1})- T (Q)| \leq |\tilde q^{k+1} - Q|$, from which we immediately deduce \eqref{EstimeesConvergence} and conclude to the convergence of Algorithm 1 for $s$ large enough. 
\end{proof}
%

\section{Conclusion}
As a conclusion, let us formulate a few comments and highlight some remaining open problems.\\

{\bf On the geometrical conditions.} Our strategy requires the use of Carleman estimates, and in particular the conditions \eqref{GCC-multiplier} and \eqref{GCC-Time}. But these conditions are much stronger than the classical {\it Geometric Control Condition} (GCC) introduced in \cite{Bardos}. Whether or not similar results as the ones presented above apply when only the GCC holds is an open problem. In particular, to our knowledge, the only stability result in inverse problem proved using micro local analysis is the recent work \cite{StefanovUhlmann2011}, which requires the GCC and the convexity of the whole boundary.\\

{\bf Smoothness of Controls.} The control process proposed in Section \ref{SubsecIntroControl} does not fit in the framework developed in \cite{ErvedozaZua10} which proves that using the Hilbert Uniqueness Method (HUM) (slightly modified by the introduction of a smooth cut-off function in time) to compute the controls, if the data to be controlled is smooth, then the corresponding control and controlled trajectory are smooth. 
Therefore, new questions arise: 

$\bullet$ Does the control process in Section \ref{SubsecIntroControl} enjoy smoothness properties similar to the ones of the classical HUM control? Note that these regularity properties arise naturally when considering the control properties of semi-linear wave equations - see \cite{DehmanLebeau2009} - or when deriving convergence rates for the discrete controls, as explained in \cite{ErvZuaCime}. 

$\bullet$ How does the usual HUM control process depend on the potentials of the wave equation? \\

{\bf Numerics and inverse problems.} Recently, in \cite{BaudouinErvedoza11}, we have proved discrete Carleman estimates for the space semi-discrete $1$-d wave equation discretized using finite differences. There, following the results on the observability of discrete waves - see e.g. \cite{ErvZuaCime} -, a new term has been added to make the Carleman estimates uniform with respect to the discretization parameter. This term, somehow corresponding to some kind of Tychonoff regularization of the Carleman estimates, is needed due to spurious waves created by the discretization process. Based on these uniform Carleman estimates, we have been able to prove a convergence result for the approximation of a potential in the inverse problem given in Section \ref{SubsecIntroInversePb}, provided a Tychonoff regularization term is added in the process. 
\\
It would then be completely natural to try to adapt the algorithm developed here in the continuous case to the space semi-discrete schemes and in numerics. This is currently under investigation. \\

\noindent{\bf Acknowledgements.} \\
The authors thank Jean-Pierre Puel and Belhassen Dehman for their encouragements concerning that work. The authors also wish to thank Institut Henri Poincar\'e (Paris, France) for providing a very stimulating environment during the ``Control of Partial and Differential Equations and Application'' program in the Fall 2010.

\bibliographystyle{plain}

\end{document}